\documentclass{amsart}

\usepackage{graphicx} 
\usepackage{amsmath}
\usepackage{amsthm}
\usepackage{amsfonts}
\usepackage{mathrsfs}
\usepackage{dsfont}
\usepackage{color}
\usepackage{enumitem}
\usepackage{tikz-cd}
\usepackage[margin=42mm]{geometry}

\newcommand{\naturals}{\mathbb{N}}

\newcommand{\angles}[1]{\left\langle #1 \right\rangle}

\newcommand{\para}[1]{\left(#1\right)}
\newcommand{\paraa}[1]{\big(#1\big)}

\newcommand{\sgn}{\operatorname{sgn}}
\newcommand{\im}{\operatorname{im}}
\newcommand{\Hom}{\operatorname{Hom}}

\newcommand{\spacearound}[1]{\quad#1\quad}
\newcommand{\equivalent}{\spacearound{\Leftrightarrow}}
\renewcommand{\implies}{\spacearound{\Rightarrow}}

\newtheorem{theorem}{Theorem}[section]
\newtheorem{corollary}[theorem]{Corollary}
\newtheorem{lemma}[theorem]{Lemma}
\newtheorem{proposition}[theorem]{Proposition}
\newtheorem{example}[theorem]{Example}
\theoremstyle{definition}
\newtheorem{definition}[theorem]{Definition}
\theoremstyle{remark}
\newtheorem{remark}[theorem]{Remark}
\numberwithin{equation}{section}

\newcommand{\A}{\mathcal{A}}
\newcommand{\K}{\mathbb{K}}

\renewcommand{\mid}{\mathds{1}} 
\renewcommand{\d}{\partial}

\newcommand{\Der}{\operatorname{Der}}
\newcommand{\qand}{\quad\text{and}\quad}

\newcommand{\g}{\mathfrak{g}}

\renewcommand{\dh}{\hat{\partial}}
\newcommand{\Omegag}[1]{\Omega^{#1}_{\g}}
\newcommand{\Omegabg}[1]{\bar{\Omega}^{#1}_{\g}}

\newcommand{\Omegaoneg}{\Omega^1_{\g}}

\newcommand{\OmegaoneD}{\Omega^1_{\D_2}}
\renewcommand{\emph}[1]{\textit{#1}}

\newcommand{\otimesK}{\otimes_{\K}}

\newcommand{\nablaz}{\nabla^0}

\newcommand{\ZA}{Z(\A)}

\renewcommand{\P}{\mathcal{P}}
\newcommand{\Pp}{\P_{+}}
\newcommand{\KQ}{\K Q}
\newcommand{\KQop}{\KQ^{\textrm{op}}}

\newcommand{\D}{\mathfrak{D}}

\newcommand{\Qp}[1]{Q_{||#1}}
\newcommand{\Qpalpha}{\Qp{\alpha}}

\newcommand{\Ppar}[1]{\P_{||#1}}
\newcommand{\Ppara}{\Ppar{\alpha}}

\newcommand{\gh}{\hat{g}}

\newcommand{\mycomment}[1]{}

\title[Derivation based noncommutative geometry of quiver path algebras]{Derivation based noncommutative geometry\\ of quiver path algebras}
\date{}

\author{Joakim Arnlind}
\address[Joakim Arnlind]{Dept. of Math.\\
Link\"oping University\\
581 83 Link\"oping\\
Sweden}
\email{joakim.arnlind@liu.se}

\author{Victor Hildebrandsson}
\address[Victor Hildebrandsson]{Dept. of Math.\\
Link\"oping University\\
581 83 Link\"oping\\
Sweden}
\email{victor.hildebrandsson@liu.se}

\begin{document}

\begin{abstract}
    In this paper, we study derivation based noncommutative geometry of path algebras arising from quivers by identifying a Lie algebra of derivations from which a differential calculus is constructed. We show that properties of the differential calculus are intimately connected with the structure of the quiver. In particular, by introducing the notion of a separating quiver, we find a condition on the quiver which is equivalent to the module of differential forms being a projective bimodule. For such quivers, one can construct bimodule connections, as well as bimodule forms which serve as noncommutative analogues of the Riemannian metric. As an illustration of the above concepts, we present a class of quivers for which one can construct Levi-Civita connections, i.e. torsion free connections compatible with a bimodule form.
\end{abstract}

\maketitle

\tableofcontents

\section{Introduction}

\noindent
Path algebras of quivers provide a rich class of noncommutative algebras, where the algebraic structure is closely tied to the properties of the underlying quiver. Such algebras are, in a sense, highly noncommutative and often have a large radical. They play a central role in the representation theory of finite-dimensional algebras, since a finite-dimensional algebra is Morita equivalent to the path algebra of a quiver with relations (see e.g \cite{ass:rep.theory.associative.algebras}). Moreover, even in the case when the path algebra admits a $\ast$-structure, it is typically not possible to embed it into a $C^\ast$-algebra, obstructing the use of $C^\ast$-algebraic methods to study their noncommutative geometry. In this paper we attempt to develop derivation based noncommutative geometry of path algebras. 

Noncommutative geometry related to graph structures have been studied in several inherently different contexts. For instance, path algebras have appeared in noncommutative algebraic geometry as important examples of formally smooth algebras; however, we point out that this approach to noncommutative geometry is quite different from the one in the current paper, as one is interested in concepts that map to their traditional geometric counterparts via the representation functor, e.g. leading to the definition of double derivations (cf. \cite{blb:necklace,cbeg:noncommutative.geometry.quiver}).  

Furthermore, graphs have been used to encode noncommutative differential geometry of the algebra of functions on the set of vertices of a graph, where the edge structure encodes the choice of differential calculus (cf. \cite{dmh:discrete.calculus.graphs,m:ncg.graphs}). 
Yet another example of noncommutative geometry on graphs comes from graph $C^\ast$-algebras, which is developed with spectral triples as a starting point (see e.g. \cite{pr:ncg.graph.Cstar}). Apart from having an analytic structure, graph $C^\ast$-algebras are closely related to Leavitt path algebras, which are constructed from path algebras by adding so called ghost edges together with a set of additional relations (cf. \cite{aasm:leavitt.path.algebras} for a comprehensive exposition). It would be interesting to study to what extent the methods developed in the current paper extend to Leavitt path algebras. 

In the following, we develop noncommutative differential calculus on the path algebra of a quiver in the context of derivation based calculus by identifying a preferred Lie algebra of (mostly) outer derivations from which a differential graded algebra is constructed (cf. \cite{dv:calculDifferentiel,a:lc.kronecker,ah:nc.levi.civita}).  
Of crucial importance is the introduction of \emph{separating quivers} which is a class of quivers where each edge is part of either zero or at least two distinct cycles. This class obviously includes all acyclic quivers, but also contains many quivers with cycles for which the corresponding path algebra is infinite dimensional. We show that separating quivers are precisely those for which the bimodule of differential forms is projective (cf. Corollary~\ref{cor:Q.separating.equiv.projective}). For such quivers we study the set of bimodules connections and via the introduction of bimodule forms (as analogues of metrics) we develop Levi-Civita bimodule connections. 

Along the way, we find more instances where the underlying structure of the quiver interacts in a non-trivial way with the geometry, for instance in the structure of endomorphisms of differential forms and in the dimension of the differential calculus.

\section{Preliminaries}

\subsection{Derivation based calculi}\label{sec:derivation.based.calculi}

\noindent
Let us start by briefly recalling the construction of noncommutative differential forms in a derivation based calculus. To this end, let $\A$ be a unital associative algebra over a field $\K$ of characteristic zero, let $\ZA$ denote the center of $\A$ and let $\Der(\A)$ denote the set of derivations of $\A$. One introduces noncommutative differential $k$-forms as follows (cf. \cite{dv:calculDifferentiel}).

\begin{definition}
    Let $\A$ be a $\K$-algebra and let $\g\subseteq\Der(\A)$ be a Lie algebra over $\K$. Define $\Omegabg{k}$ be the set of alternating $\K$-multilinear maps $\omega:\g^k\to\A$, i.e.
    \begin{align*}
        &\omega(\d_1,\ldots,\d_k) = \sgn(\sigma)\omega(\d_{\sigma(1)},\ldots,\d_{\sigma(k)})\\
        &\omega(\lambda\d_1,\d_2,\ldots,\d_k) = \lambda\omega(\d_1,\d_2,\ldots,\d_k)
    \end{align*}
    for $\sigma\in S_k$, $\lambda\in\K$ and $\{\d_1,\ldots,\d_k\}\subseteq\g$ (where $S_k$ denotes the symmetric group of order $k$). 
    Furthermore, we set
    \begin{align*}
        \Omegabg{} = \bigoplus_{k\geq 0}\Omegabg{k},      
    \end{align*}
    with $\Omegabg{0}=\A$.
\end{definition}

\noindent
Note that the dimension of the Lie algebra $\g$ bounds the maximum order of nonzero differential forms, due to the antisymmetry property. Namely, if $\dim(\g)=n$ then $\Omegabg{k}=0$ for $k>n$. In general, we define the dimension of $\Omegabg{}$ to be the largest positive integer $n$ such that $\Omegabg{n}\neq 0$ and $\Omegabg{k}=0$ for all $k>n$.

For $\omega\in\Omegabg{k}$ and
$\tau\in\Omegabg{l}$ one defines $\omega\tau\in\Omegabg{k+l}$ as
\begin{align*}
  (\omega\tau)(\d_1,\ldots,\d_{k+l}) =
  \frac{1}{k!l!}\sum_{\sigma\in S_{k+l}}\sgn(\sigma)\omega(\d_{\sigma(1)},\ldots,\d_{\sigma(k)})
  \tau(\d_{\sigma(k+1)},\ldots,\d_{\sigma(k+l)}),
\end{align*}
and introduces
$d_k:\Omegabg{k}\to\Omegabg{k+1}$ as $(d_0a)(\d_0)=\d_0(a)$ for
$a\in\Omegabg{0}=\A$ and
\begin{align*}
  d_k\omega(\d_0,\ldots,\d_{k}) =
  &\sum_{i=0}^k(-1)^i\d_i\paraa{\omega(\d_0,\ldots,\hat{\d}_{i},\ldots,\d_k)}\\
  &\qquad+\sum_{0\leq i<j\leq k}(-1)^{i+j}\omega\paraa{[\d_i,\d_j],\d_0,\ldots,\hat{\d}_i,\ldots,\hat{\d}_j,\ldots,\d_k},
\end{align*}
for $\omega\in\Omegabg{k}$ with $k\geq 1$, satisfying $d_{k+1}d_{k}=0$, where $\hat{\d}_i$ denotes the omission
of $\d_i$ in the argument. When there is no risk of confusion, we
shall omit the index $k$ and simply write
$d:\Omegabg{k}\to\Omegabg{k+1}$. For instance, if
$\omega\in\Omegabg{1}$ then
\begin{align*}
  d\omega(\d_0,\d_1) = \d_0\omega(\d_1)-\d_1\omega(\d_0)-\omega([\d_0,\d_1])
\end{align*}
and, in general, we say that $\omega\in\Omegabg{k}$ is closed if $d\omega=0$, and exact if there exists $\eta\in\Omegabg{k-1}$ such that $\omega=d\eta$. 
Moreover, we note that the graded product rule is satisfied
\begin{align*}
  d(\omega\eta) = (d\omega)\eta + (-1)^k\omega d\eta 
\end{align*}
for $\omega\in\Omegabg{k}$ and $\eta\in\Omegabg{l}$. The above definitions make $\Omegabg{}$ into a differential graded algebra with $\deg(\omega)=k$ if $\omega\in\Omegabg{k}$.
The $\A$-bimodule structure of $\Omegabg{k}$ is given by
\begin{align*}
  &(a\omega)(\d_1,\ldots,\d_k) = a\omega(\d_1,\ldots,\d_k)\\
  &(\omega a)(\d_1,\ldots,\d_k) = \omega(\d_1,\ldots,\d_k)a
\end{align*}
for $a\in\A$, $\omega\in\Omegabg{k}$.
Furthermore, for $\d\in\g$ one introduces
\begin{align*}
    &\iota_{\d}:\Omegabg{k+1}\to\Omegabg{k}\\
    &\iota_{\d}(\omega)(\d_1,\ldots,\d_k) = \omega(\d,\d_1,\ldots,\d_k)\quad(k\geq 1)\\
    &\iota(a) = 0 \quad\text{for}\quad a\in\Omegabg{0}=\A
\end{align*}
and
\begin{align*}
    &L_{\d}:\Omegabg{k}\to\Omegabg{k}\qquad
    L_{\d} = d\circ\iota_{\d}+\iota_{\d}\circ d 
\end{align*}
satisfying 
\begin{align*}
    &\iota_{\d}(\omega\eta) = \iota_{\d}(\omega)\eta + (-1)^k\omega\,\iota_{\d}(\eta)\\
    &L_{\d}(\omega\eta) = L_{\d}(\omega)\eta + \omega L_{\d}(\eta)
\end{align*}
for $\omega\in\Omegabg{k}$ and $\eta\in\Omegabg{l}$.

In the following, we shall primarily be interested in a subalgebra $\Omegag{}\subseteq\Omegabg{}$, generated by the algebra $\A$; that is, one sets $\Omegag{0}=\A$ and
\begin{align*}
   \Omegag{k} = \{a_0da_1da_2\cdots da_k: a_i\in\A\text{ for }i=0,\ldots,k\}
\end{align*}
for $k\geq 1$. It is straightforward to check that $\Omegag{}=\bigoplus_{k\geq 0}\Omegag{k}$ is a differential graded subalgebra of $\Omegabg{}$. The differential graded algebra $\Omegag{}$ more closely resembles the algebra of differential forms on a manifold, where the differential forms are generated by the algebra of smooth functions.
Moreover, $(\Omega^1,d)$ is a first order differential calculus since, by definition, $\Omega^1$ is generated by $\A$ (see e.g. \cite{w:diff.calc.matrix.pseudo}).

Let us now recall the theory of connections and bimodule forms with resect to $\Omegaoneg$.

\begin{definition}\label{def_connection}
    Let $\A$ be a $\K$-algebra and let $\g\subseteq\Der(\A)$ be a Lie algebra over $\K$. A left (right) $\g$-connection on a left (right) $\A$-module $M$ is a map $\nabla:M\times\g\to M$ such that
    \begin{enumerate}
        \item
        $\nabla_{\d}(m+m')=\nabla_{\d} m
        +\nabla_{\d} m'$,
        
        \item
        $\nabla_{\lambda\d+\lambda'\d'}m
        =\lambda\nabla_\d m+\lambda'\nabla_{\d'}m$,
        
        \item
        $\nabla_\d(am)=a(\nabla_\d m)+(\d a)m$,
        \qquad $\paraa{\nabla_{\d}(ma)=(\nabla_d m)a+m(\d a)}$
    \end{enumerate}
    for $m,m'\in M$, $\d,\d'\in\mathfrak{g}$, $\lambda,\lambda'\in\K$ and $a\in\mathcal{A}$. Furthermore, a bimodule $\g$-connection on a $\A$-bimodule $M$ is a left $\g$-connection on $M$ which is also a right $\g$-connection.
\end{definition}

\noindent
In differential geometry, Riemannian metrics represent inner products of sections of the tangent bundle (and, dually, on sections of the cotangent bundle). Bimodule forms play the role of a metric structure on bimodules over the algebra. 

\begin{definition}
    Let $\A$ be a $\K$-algebra and let $M$ be an $\A$-bimodule. A bimodule form $g$ on $M$ is a bimodule homomorphism
    \begin{align*}
        g:M\otimesK M\to \A.
    \end{align*}
    Moreover, a strong bimodule form $g$ on $M$ is a bimodule homomorphism
    \begin{align*}
        g: M\otimes_{\A} M\to \A.
    \end{align*}
\end{definition}

\noindent
For notational convenience, we shall often write $g(a\otimesK b)=g(a,b)$.

As in Riemannian geometry, connections can be compatible with the metric in the sense that it preserves inner products of parallel transported vectors. In the noncommutative setting, it is straight forward to introduce an analogous concept. 

\begin{definition}
    Let $\A$ be a $\K$-algebra and let $M$ be an $\A$-bimodule with a bimodule $\g$-connection $\nabla$. Given a bimodule form $g$ on $M$, the connection $\nabla$ is compatible with $g$ if
    \begin{align}
        \d g(m_1,m_2) = g(\nabla_{\d} m_1,m_2) + g(m_1,\nabla_{\d}m_2)
    \end{align}
    for $m_1,m_2\in M$ and $\d\in\g$.
\end{definition}

\noindent 
By definition, a Riemannian metric is invertible in the sense that it provides a bijection between vector fields and differential forms. In noncommutative geometry, this is in general a quite restrictive condition and it is useful to have more relaxed notions of nondegeneracy for bimodule forms.

In the following, for an $\A$-bimodule $M$, we let ${}_{\A}M$ and $M_{\A}$ denote $M$ considered as a left and right $\A$-module, respectively.

\begin{definition}
    Let $\A$ be a $\K$-algebra and let $M$ be an $\A$-bimodule. Given a bimodule form $g$ on $M$ one defines
    \begin{align*}
        &\gh_L:M\to ({}_{\A}M)^\ast = \Hom_{\A}({}_{\A}M,\A)\qquad \gh_L(m)(n) = g(n,m)\\
        &\gh_R:M\to (M_{\A})^\ast = \Hom_{\A}(M_{\A},\A)\qquad \gh_R(m)(n) = g(m,n).
    \end{align*}
    Furthermore, $g$ is called left (right) nondegenerate if $\gh_L$ ($\gh_R$) is injective and $g$ is called left (right) invertible if $\gh_L$ ($\gh_R$) is a bijection.
\end{definition}

\noindent Recall that a module $M$ is reflexive if the canonical map $\iota: M\to M^{\ast\ast}$ defined by
\begin{align*}
    \iota(m)(\phi) = \phi(m)\qquad (m\in M,\phi\in M^{\ast})
\end{align*}
is bijective. For bimodules that are reflexive as both left and right modules, left and right invertibility is equivalent, and we simply say that the bimodule form is invertible.

\begin{proposition}
    Let $\A$ be a $\K$-algebra, let $M$ be an $\A$-bimodule and let $g$ be a bimodule form on $M$. If $M$ is reflexive as both a left and right module then
    \begin{align*}
        \gh_L\text{ bijective}\equivalent \gh_R\text{ bijective}.
    \end{align*}
\end{proposition}

\begin{proof}
    Assume that $\gh_L$ is bijective, implying that the dual map
    \begin{align*}
        \gh_L^\ast : &({}_{\A}M)^{\ast\ast} \to ({}_{\A}M)^{\ast}\\
        f &\mapsto \gh_L^\ast(f)(m) = f(\gh_L(m))
    \end{align*}
    is bijective as well. Now, let $\iota:{}_{\A}M\to ({}_{\A}M)^{\ast\ast}$ denote the canonical map and note that $\gh_L^\ast\circ\iota:{}_{\A}M\to({}_{\A}M)^\ast$ becomes
    \begin{align*}
        \paraa{(\gh_L^\ast\circ\iota)(m)}(n)
        &= \iota(m)(\gh_L(n)) = \paraa{\gh_L(n)}(m)
        =g(m,n) = \paraa{\gh_R(m)}(n).
    \end{align*}
    In other words, $\gh_R = \gh_L^\ast\circ\iota$. Since ${}_{\A}M$ is assumed to be reflexive as a left module it follows that $\iota$, and consequently, $\gh_R$ is bijective. In the same way, when considering $M$ as a right module, one shows that if $\gh_R$ is bijective then $\gh_L$ is bijective. 
\end{proof}

\noindent
Note that since finitely generated projective modules are reflexive the above result applies to finitely generated projective bimodules. 

In the above results, we have considered connections and bimodule forms on arbitrary modules, corresponding to the geometric situation of vector bundles (in case the module is a finitely generated projective module). In this paper, we shall mostly be interested in connections on the bimodule $\Omegaoneg$, for which there is a natural concept of torsion in analogy with differential geometry.

\begin{definition}
    Let $\nabla$ be a left (right) $\g$-connection on $\Omegaoneg$. The torsion $T(\nabla):\Omegaoneg\to\Omegabg{2}$ is defined as
    \begin{align}\label{eq:torsion.def}
        \paraa{T(\nabla)\omega}(\d_1,\d_2) =
        (\nabla_{\d_1}\omega)(\d_2)-(\nabla_{\d_2}\omega)(\d_1)-d\omega(\d_1,\d_2) 
    \end{align}
    for $\omega\in\Omegaoneg$ and $\d_1,\d_2\in\g$. Moreover, a connection is called torsion free if $T(\nabla)=0$. 
\end{definition}

\begin{remark}
Although it will not be important in what follows, we point out that $T(\nabla)\omega$ is an element of $\Omegabg{2}$ but not necessarily an element of $\Omega^2_{\g}$. Connections for which $T(\nabla)\omega\in\Omegag{2}$ are called regular, and it follows that torsion free connections are regular (see \cite{ah:nc.levi.civita} for a detailed discussion of regular connections).    
\end{remark}

\noindent
Let $\nabla,\nabla'$ be bimodule $\g$-connections on $\Omegaoneg$. It is a standard fact that 
\begin{align*}
    A(\omega,\d) = \nabla_{\d}\omega-\nabla'_{\d}\omega
\end{align*}
for $\omega\in\Omegaoneg$ and $\d\in\g$ is a an element of $\Hom_{\A,\K}(\Omegaoneg\times\g,\Omegaoneg)$, i.e. a bimodule homomorphism in the first argument and a vector space map in the second. Conversely, given a bimodule $\g$-connection $\nabla^0$ on $\Omegaoneg$, any connection $\nabla$ can be written as
\begin{align*}
    \nabla = \nabla^0 + A
\end{align*}
for some $A\in\Hom_{\A,\K}(\Omegaoneg\times\g,\Omegaoneg)$. Moreover, for any $L\in\Hom_{\K,\K}(\Omegaoneg\times\g,\Omegaoneg)$ one introduces $\wedge L:\Omegaoneg\to\Omegabg{2}$
\begin{align*}
    \wedge L(\omega)(\d_1,\d_2) 
    = L(\omega,\d_1)(\d_2)-L(\omega,\d_2)(\d_1)
\end{align*}
implying that the definition of torsion in \eqref{eq:torsion.def} can be written as
\begin{align*}
    T(\nabla)(\omega) = \wedge\nabla(\omega) - d\omega. 
\end{align*}
Correspondingly, if $\nabla$ is a torsion free connection and 
$A\in\Hom_{\A,\K}(\Omegaoneg\times\g,\Omegaoneg)$ then $\nabla+A$ is torsion free if and only if $\wedge A=0$.

\subsection{Quivers and path algebras}

In the remaining part of the paper, we will be interested in noncommutative geometry over path algebras of quivers. Therefore, let us recall a few basic properties and introduce the relevant notation that will be used in the following.  

\begin{definition}
    A quiver $Q$ is a tuple $Q=(Q_0,Q_1,s,t)$ where $Q_0,Q_1$ are sets and $s,t$ are maps $Q_1\to Q_0$. Elements of $Q_0$ are called vertices and elements of $Q_1$ are called edges (or arrows). For $\alpha\in Q_1$, $s(\alpha)$ is called its source and $t(\alpha)$ is called its target. For brevity, we shall often write $Q=(Q_0,Q_1)$.
\end{definition}

\noindent
A quiver $Q=(Q_0,Q_1)$ is called finite if $|Q_0|<\infty$ and $|Q_1|<\infty$. In the following, we assume all quivers to be finite. 

Let $Q=(Q_0,Q_1)$ be a quiver. A path $p$ in $Q$ is a sequence of edges $p=\alpha_1\alpha_2\cdots\alpha_n$ with $\alpha_k\in Q_1$ such that $t(\alpha_k)=s(\alpha_{k+1})$ for $k=1,\ldots,n-1$, and we let $l(p)=n$ denote the length of the path. By convention, a vertex $v\in Q_0$ is a path of length zero and we let $\P(Q)$ denote the set of all paths. For $\alpha\in Q_1$ and $p\in\P(Q)$ we write $\alpha\in p$ if there exists $p_1,p_2\in\P(Q)$ such that $p=p_1\alpha p_2$. Furthermore, we set
\begin{align*}
    &\Pp(Q) = \{p\in\P(Q):l(p)\geq 1\}\\
    &\P_k(Q) = \{p\in \P(Q): l(p)=k\}
\end{align*}
for $k\geq 0$. One extends $s,t$ to maps $\P(Q)\to Q_0$ by setting
\begin{align*}
    &s(\alpha_1\cdots\alpha_n)=s(\alpha_1)\\
    &t(\alpha_1\cdots\alpha_n)=t(\alpha_n),
\end{align*}
and we say that $p,q\in\P(Q)$ are parallel, and write $p||q$, if $s(p)=s(q)$ and $t(p)=t(q)$; moreover, for $q\in\P(Q)$ we introduce the set of all paths parallel to $q$
\begin{align*}
    \Ppar{q} = \{p\in\P(Q): p||q\},
\end{align*}
and for $\alpha\in Q_1$
\begin{align*}
    \Qpalpha = \{\beta\in Q_1:\exists p\in\Ppar{\alpha}: \beta\in p\}.
\end{align*}
For $u,v\in Q_0$ we set
\begin{align*}
    \P(u,v) = \{p\in\P(Q):s(p)=u,\, t(p)=v\}.
\end{align*}
A path $p\in\Pp(Q)$ is called a cycle if $s(p)=t(p)$; a cycle of length one is called a loop. Two cycles $c_1,c_2\in\Pp(Q)$ are called distinct if they are not powers of the same cycle. A quiver is called acyclic if it contains no cycles. 
Furthermore, we let 
\begin{align*}
    \P_A=\{p\in\P(Q):s(p)\neq t(p)\}   
\end{align*}
denote the set of acyclic paths. A quiver is called connected if the corresponding undirected graph of the quiver is connected. 

Finally, for $u\in Q_0$ and $\alpha\in Q_1$, we introduce
\begin{align*}
    &T_u = \{p\in\P(Q):t(p)=u\}\qquad
    S_u = \{q\in\P(Q):s(q)=u\}\\
    &T_\alpha = T_{s(\alpha)}\qquad S_{\alpha} = S_{t(\alpha)}\qquad
    B_{\alpha} = T_{\alpha}\times S_{\alpha}.
\end{align*}
Thus, $T_{\alpha}$ is the set of paths ending at $s(\alpha)$ and $S_{\alpha}$ is the set of paths starting at $t(\alpha)$.

A slightly less standard concept for quivers is that of a bypass. An edge is a bypass if it connects the source and target of a (possibly) longer path. This concept will later be important when considering endomorphisms of $\Omegaoneg$. Let us recall the precise definition.

\begin{definition}[cf. \cite{acmt:simply.connected.shurian}]
    Let $Q=(Q_0,Q_1)$ be a quiver. For $\alpha\in Q_1$ and $p\in\P(Q)$, the pair $(\alpha,p)$ is called a bypass if $p\in\Ppara\cap\Pp(Q)$ and $p\neq\alpha$.  An edge $\alpha\in Q_1$ is called a bypass if there exists a bypass $(\alpha,p)$ or, in other words, if $|\Ppara|>1$. If $\alpha$ is not a bypass for any $\alpha\in Q_1$ then $Q$ is called a nonpassing quiver. 
\end{definition}

\noindent
Given a quiver $Q=(Q_0,Q_1)$ one introduces its path algebra as the vector space of paths together with the concatenation product. 

\begin{definition}
    Let $Q=(Q_0,Q_1,s,t)$ be a quiver and let $\K$ be a field. The path algebra $\K Q$ is the $\K$-vector space generated by the elements of $\P(Q)$ together with the multiplication
    \begin{align*}
        &up = \delta_{u,s(p)}p\qquad pu = \delta_{t(p),u}p
        \qquad\text{for }u\in Q_0,p\in\P(Q)\\
        &(\alpha_1\cdots\alpha_n)(\beta_1\cdots\beta_m) 
        = \delta_{t(\alpha_n),s(\beta_1)}\alpha_1\cdots\alpha_n\beta_1\cdots\beta_m.
    \end{align*}
    By definition, $\P(Q)$ is a basis of $\K Q$ as a vector space over $\K$.
\end{definition}

\noindent
The path algebra of a quiver is an $\naturals$-graded algebra graded by path length, i.e.
\begin{align*}
    \KQ = \bigoplus_{k\geq 0}\KQ_{k}.
\end{align*}
where $\KQ_k\subseteq \KQ$ denotes the subspace generated by paths of length $k$. 

Since elements of $Q_0$ are idempotents, one may introduce the left projective $\KQ$-module
\begin{align*}
    &P_u = \KQ u
\end{align*}
for $u\in Q_0$, as well as the projective $\KQ$-bimodule
\begin{align*}
    P_{u,v} = \KQ u\otimesK v\KQ
\end{align*}
for $u,v\in Q_0$. In the standard way, we equivalently consider $\KQ$-bimodules as left $\KQ^e$-modules, where $\KQ^e=\KQ\otimesK\KQop$ denotes the enveloping algebra of $\KQ$, giving
\begin{align*}
    P_{u,v} = (\KQ\otimesK \KQop)(u\otimesK v).
\end{align*}

\noindent
In particular, we will be interested in the following bimodules
\begin{align*}
    P_{\alpha} = P_{s(\alpha),t(\alpha)}
\end{align*}
for $\alpha\in Q_1$. Elements of $P_{\alpha}$ are sums of tensor products of the form $p\otimesK q$ with $t(p)=s(\alpha)$ and $s(q)=t(\alpha)$, i.e. $p\in T_{\alpha}$ and $q\in S_{\alpha}$. Hence, every element $x\in P_{\alpha}$ can be written as
\begin{align*}
    x = \sum_{(p,q)\in B_{\alpha}}\lambda_{(p,q)}p\otimesK q
\end{align*}
with $\lambda_{(p,q)}\in\K$. In fact, since $\P(Q)$ is a (vector space) basis of $\KQ$ it follows that $\{p\otimesK q:(p,q)\in B_{\alpha}\}$ is a basis of $P_{\alpha}$, implying that the coefficients $\lambda_{(p,q)}$ are unique. 

\subsection{Derivations on path algebras}

A derivation based differential calculus is defined starting from a choice of Lie algebra of derivations. In the following, we identify a particular Lie algebra that will be used to construct the differential calculus.

In \cite{gl:structure.hochschild.path.algebras} there is a complete description of the Lie algebra of derivations of a path algebra. However, let us recall a particular subalgebra of derivations that will be of interest to us and provide a self contained account of the basic results.

Let $Q=(Q_0,Q_1)$ be a quiver. For $\alpha\in Q_1$ and $a\in\KQ$ define $D_{\alpha,a}:\KQ\to\KQ$ as
\begin{align*}
    &D_{\alpha,a}(v) = 0\qquad\text{for }v\in Q_0\\
    &D_{\alpha,a}(\alpha_1\cdots\alpha_n)
    =\sum_{k=1}^n\delta_{\alpha,\alpha_k}
    \alpha_1\cdots\alpha_{k-1}a\alpha_{k+1}\cdots\alpha_n
\end{align*}
and extended it to $\KQ$ by linearity. It follows immediately that
\begin{align}
    D_{\alpha,\lambda a+b} = \lambda D_{\alpha,a}+D_{\alpha,b}
\end{align}
for $\alpha\in Q_1$, $\lambda\in\K$ and $a,b\in\KQ$. Note that a necessary condition for the above map to be nontrivial is that $a$ contains paths that are parallel to $\alpha$. Let us now show that $D_{\alpha,a}$ are indeed derivations of $\KQ$ closed under the Lie bracket.

\begin{proposition}\label{prop:Dalpha.properties}
    If $\alpha\in Q_1$ and $a\in\KQ$ then
    \begin{align}\label{eq:Dalpha.derivation}
        D_{\alpha,a}(b_1b_2)=D_{\alpha,a}(b_1)b_2+b_1D_{\alpha,a}(b_2)
    \end{align}
    for $b_1,b_2\in\KQ$ and
    \begin{align}\label{eq:Dalpha.commutator}
        [D_{\alpha,a},D_{\beta,b}] = 
        D_{\beta,D_{\alpha,a}(b)}-D_{\alpha,D_{\beta,b}(a)}.
    \end{align}
    for $\alpha,\beta\in Q_1$ and $a,b\in\KQ$.
\end{proposition}

\begin{proof}
    Let us start by showing that \eqref{eq:Dalpha.derivation} holds, and we note that it is sufficient to prove it for $b_1,b_2\in\P(Q)$. First, we note that if $b_1b_2=0$ then $t(b_1)\neq s(b_2)$ and $D_{\alpha,a}(b_1)b_2=b_1D_{\alpha,a}(b_2)=0$ since every path in $D_{\alpha,a}(b_k)$ has the same source and target as $b_k$ for $k=1,2$. Now, assume that $b_1,b_2\in\P(Q)$ and $b_1b_2\neq 0$ (implying that $t(b_1)=s(b_2)$). If $b_1\in Q_0$ then 
    \begin{align*}
        D_{\alpha,a}(b_1)b_2 + b_1D_{\alpha,a}(b_2)
        =b_1D_{\alpha,a}(b_2) = D_{\alpha,a}(b_1b_2)
    \end{align*}
    since $D_{\alpha,a}(b_1)=0$, and similarly if $b_2\in Q_0$; thus, we assume in the following that $b_1,b_2\in\P_+(Q)$ and write
    \begin{align*}
        &b_1 = \alpha_1\cdots\alpha_n\qquad 
        b_2= \beta_1\cdots\beta_m\\
        &b_1b_2=\alpha_1\cdots\alpha_n\beta_1\cdots\beta_m
        =\gamma_1\cdots \gamma_{n+m}
    \end{align*}
    for $n,m\geq 1$. A straight forward computation yields
    \begin{align*}
        &D_{\alpha,a}(b_1)b_2 + b_1D_{\alpha,a}(b_2)\\
        &= \sum_{k=1}^n\delta_{\alpha,\alpha_k}\alpha_{1}\cdots\alpha_{k-1}a\alpha_{k+1}\cdots\alpha_n\beta_1\cdots\beta_m
        +\sum_{k=1}^m\delta_{\alpha,\beta_k}\alpha_1\cdots\alpha_n\beta_1\cdots\beta_{k-1}a\beta_{k+1}\cdots\beta_m\\
        &=\sum_{k=1}^{m+n}\delta_{\alpha,\gamma_k}\gamma_1\cdots\gamma_{k-1}a\gamma_{k+1}\cdots\gamma_{m+n}
        =D_{\alpha,a}(\gamma_1\cdots\gamma_{m+n})
        =D_{\alpha,a}(b_1b_2),
    \end{align*}
    which completes the proof of \eqref{eq:Dalpha.derivation}. 
    
    Next, let us prove \eqref{eq:Dalpha.commutator} by showing that it holds when acting on arbitrary $p\in\P(Q)$. First, if $p\in Q_0$ then every term vanishes since $D_{\gamma,c}(v)=0$ for all $v\in Q_0$, $\gamma\in Q_1$ and $c\in \P(Q)$. Now, assume that $p\in\P_+(Q)$ and let us proceed by induction. For $p=\gamma\in\P_1(Q)$ one finds that
    \begin{align*}
        [D_{\alpha,a},D_{\beta,b}](\gamma)
        &= \delta_{\beta,\gamma}D_{\alpha,a}(b)-\delta_{\alpha,\gamma}D_{\beta,b}(a)
        =D_{\beta,D_{\alpha,a}(b)}(\gamma)-D_{\alpha,D_{\beta,b}(a)}(\gamma)
    \end{align*}
    showing \eqref{eq:Dalpha.commutator} in this case. Now assume that \eqref{eq:Dalpha.commutator} holds for all $p\in\P_k(Q)$ with $k\leq n$.
    Using \eqref{eq:Dalpha.derivation} one writes
    \begin{align*}
        [D_{\alpha,a},D_{\beta,b}](\gamma_1\cdots\gamma_{n+1})
        &= [D_{\alpha,a},D_{\beta,b}](\gamma_1\cdots\gamma_n)\gamma_{n+1}
        +\gamma_1\cdots\gamma_n[D_{\alpha,a},D_{\beta,b}](\gamma_{n+1})
    \end{align*}
    and using that \eqref{eq:Dalpha.commutator} holds for paths of length at most $n$, one obtains
    \begin{align*}
        [D_{\alpha,a},D_{\beta,b}](\gamma_1\cdots\gamma_{n+1})
        &=\paraa{D_{\beta,D_{\alpha,a}(b)}(\gamma_1\cdots\gamma_n)-D_{\alpha,D_{\beta,b}(a)}(\gamma_1\cdots\gamma_n)}\gamma_{n+1}\\
        &\qquad+\gamma_1\cdots\gamma_n\para{D_{\beta,D_{\alpha,a}(b)}(\gamma_{n+1})-D_{\alpha,D_{\beta,b}(a)}(\gamma_{n+1})}\\
        &=D_{\beta,D_{\alpha,a}(b)}(\gamma_1\cdots\gamma_{n+1})-D_{\alpha,D_{\beta,b}(a)}(\gamma_1\cdots\gamma_{n+1}).
    \end{align*}
    It follows by induction that \eqref{eq:Dalpha.commutator} holds for all $p\in\P(Q)$.
\end{proof}

\noindent
It follows from  Proposition~\ref{prop:Dalpha.properties} that
\begin{align*}
    \D_2 = \{D_{\alpha,a}:\alpha\in Q_1,\, a\in\KQ\}
\end{align*}
is a Lie algebra over $\K$ and we note that $\D_2$ is generated by $D_{\alpha,p}$ for $\alpha\in Q_1$ and $p\in\P(Q)$ such that $\alpha||p$. In fact, $\{D_{\alpha,p}:\alpha\in Q_1,\,p\in \Ppara\}$ is a basis for $\D_2$ since
\begin{align*}
    &\sum_{\alpha\in Q_1,p\in \Ppara}\hspace{-4mm} c_{\alpha,p}D_{\alpha,p}=0\implies
    \sum_{\alpha\in Q_1,p\in \Ppara}\hspace{-4mm} c_{\alpha,p}D_{\alpha,p}(\beta)=0\quad\forall\,\beta\in Q_1\implies\\
    &\sum_{p\in \Ppar{\beta}}\hspace{-1mm}c_{\beta,p}p=0\quad\forall\,\beta\in Q_1
    \implies
    c_{\beta,p}=0\quad\forall\,\beta\in Q_1,p\in\Ppar{\beta},
\end{align*}
showing that these derivations are indeed linearly independent. In particular, since $D_{\alpha,\alpha}\in\D_2$ for all $\alpha\in Q_1$ one concludes that $\dim_{\K}(\D_2)\geq|Q_1|$.

In \cite{gl:structure.hochschild.path.algebras} it is shown that 
\begin{align*}
    \{[p,\cdot]:p\in\P_A(Q)\}\cup\{D_{\alpha,p}:\alpha\in Q_1,\,p\in \P(Q)\}
\end{align*}
is a basis of $\Der(\KQ)$. In particular, $\D_2$ contains all outer derivations, and no inner derivations defined by acyclic paths. However, in \cite{gl:structure.hochschild.path.algebras} it is shown that for a cycle, i.e. $p\in\P(Q)$ such that $s(p)=t(p)=v$, one obtains
\begin{align*}
    [p,\cdot] = \hspace{-4mm}\sum_{\alpha\in Q_1:s(\alpha)=v}\hspace{-4mm}D_{\alpha,p\alpha}
    \hspace{2mm}-\hspace{-2mm}\sum_{\alpha\in Q_1:t(\alpha)=v}\hspace{-4mm}D_{\alpha,\alpha p},
\end{align*}
implying that $\D_2$ contains all inner derivations defined by cycles. In particular, $\D_2$ always contains the inner derivations $[v,\cdot]$ for $v\in Q_0$, written as
\begin{align}\label{eq:inner.der.in.D2}
    [v,\cdot] = \hspace{-4mm}\sum_{\alpha\in Q_1:s(\alpha)=v}\hspace{-4mm}D_{\alpha,\alpha}
    \hspace{2mm}-\hspace{-2mm}\sum_{\beta\in Q_1:t(\beta)=v}\hspace{-4mm}D_{\beta,\beta}.
\end{align}

\section{Noncommutative differential calculi on path algebras}

\noindent 
In the following, we will study derivation based differential calculi on path algebras generated by the Lie algebra $\D_2$. Since $\d u=0$ for all $\d\in\D_2$ and $u\in Q_0$, it follows that $du=0$ and $\OmegaoneD$ is generated (as a bimodule) by $\{d\alpha:\alpha\in Q_1\}$. The form $d\alpha$ satisfy relations that have a close similarity with those of the edge $\alpha$; let us summarize this in the following proposition. 

\begin{proposition}\label{prop:OmegaoneD.relations}
    Let $Q=(Q_0,Q_1)$ be a quiver. Then $\OmegaoneD$ is generated as a $\KQ$-bimodule by $\{d\alpha:\alpha\in Q_1\}$ and
    \begin{align*}
        &p(d\alpha) = 0\equivalent t(p)\neq s(\alpha)\\
        &(d\alpha)q = 0\equivalent t(\alpha)\neq s(q)
    \end{align*}
    for $\alpha\in Q_1$ and $p,q\in\P(Q)$. Moreover,
    \begin{align*}
        s(\alpha)d\alpha = d\alpha\, t(\alpha) = d\alpha.
    \end{align*}
    for $\alpha\in Q_1$.
\end{proposition}

\begin{proof}
    Since the differential satisfies Leibniz' rule and $da$, for $a\in\KQ$, generate $\OmegaoneD$, it is clear that $\{d\alpha:\alpha\in Q_1\}$ generates $\OmegaoneD$ as a bimodule. Now, assume that $p(d\alpha)=0$ for some $p\in\P(Q)$. Evaluating this form on the derivation $D_{\alpha,\alpha}\in\D_2$ gives $p\alpha=0$ which implies that $t(p)\neq s(\alpha)$. Conversely, assume that $t(p)\neq s(\alpha)$ and let us show that $p(d\alpha)(D_{\beta,q})=0$ for all $\beta\in Q_1$ and $q\in\P(Q)$ such that $\beta||q$.
    One finds that
    \begin{align*}
        p(d\alpha)(D_{\alpha,q}) = pq = 0
    \end{align*}
    since $t(p)\neq s(\alpha)=s(q)$, and
    \begin{align*}
        p(d\alpha)(D_{\beta,q}) = pD_{\beta,q}(\alpha) = 0
    \end{align*}
    for $\beta\neq \alpha$; thus, $p(d\alpha)(\d)=0$ for all $\d\in\D_2$. Similarly, one proves that $(d\alpha)q=0$ if and only if $t(\alpha)\neq s(q)$.

    Next, let us compute
    \begin{align*}
        &\paraa{s(\alpha)d\alpha}(D_{\beta,p}) = 
        \begin{cases}
            0 &\text{if }\alpha\neq\beta\\
            s(\alpha)p=p &\text{if }\alpha=\beta
        \end{cases}
        \quad=d\alpha(D_{\beta,p})\\
        &\paraa{d\alpha\,t(\alpha)}(D_{\beta,p}) = 
        \begin{cases}
            0 &\text{if }\alpha\neq\beta\\
            pt(\alpha)=p &\text{if }\alpha=\beta
        \end{cases}
        \quad=d\alpha(D_{\beta,p})
    \end{align*}
    since $p||\beta$. This shows that $s(\alpha)d\alpha=d\alpha\,t(\alpha)=d\alpha$.
\end{proof}

\noindent 
In general, the dimension of $\Omega_{\g}$ is bounded by $\dim(\g)$. However, for $\Omega_{\D_2}$ the dimension is intimately connected to the longest path in the quiver, as formulated in the next result for acyclic quivers. 

\begin{proposition}
    Let $Q=(Q_0,Q_1)$ be an acyclic quiver. Then $\Omega_{\D_2}^n\neq 0$ if and only if there exists a path of length $n$.
\end{proposition}

\begin{proof}
    Let $p=\alpha_1\cdots\alpha_n$ be a path of length $n$; since the quiver is acyclic every edge in the path is distinct. Then 
    \begin{align*}
        d\alpha_1d\alpha_2\cdots d\alpha_n(D_{\alpha_1,\alpha_1},D_{\alpha_2,\alpha_2},\ldots,D_{\alpha_n,\alpha_n}) = \alpha_1\alpha_2\cdots\alpha_n\neq 0
    \end{align*}
    showing that there exists a nonzero element in $\Omega_{\D_2}^n$. Conversely, assume that there exists a nonzero element $\omega\in\Omega_{\D_2}^n$. Let us write
    \begin{align*}
        \omega = \sum_{k=1}^N \lambda_k p_0^k(dp_1^k)(dp_2^k)
        \cdots (dp_n^k),
    \end{align*}
    with $p_0^k\in\P(Q)$ and $p_l^k\in\Pp(Q)$ for $0\leq l\leq n$ and $1\leq k\leq N$. Since $\omega\neq 0$ there exists 
    $D_{\alpha_1,q_1},\ldots,D_{\alpha_n,q_n}$, $1\leq k_0\leq N$ and $\sigma\in S_n$ such that
    \begin{align*}
        p_0^{k_0} dp_1^{k_0}(D_{\alpha_{\sigma(1)},q_{\sigma(1)}})\cdots 
        dp_{n}^{k_0}(D_{\alpha_{\sigma(n)},q_{\sigma(n)}})\neq 0
    \end{align*}
    implying that there exists a path of length at least $n$ and, hence, a path of length $n$.
\end{proof}

\noindent
For an acyclic quiver $Q$, the above result implies that the dimension of the differential calculus is given by the length of the longest path in $Q$.

For each $\alpha\in Q_1$, let $\angles{d\alpha}$ denote the bimodule generated by $d\alpha$. The next result shows that $\OmegaoneD$ splits into a direct sum of such submodules. 

\begin{proposition}\label{prop:Omega.decomp.dalpha}
    Let $Q=(Q_0,Q_1)$ be a quiver. Then
    \begin{align*}
        \phi:\bigoplus_{\alpha\in Q_1}\angles{d\alpha}\to \OmegaoneD\quad\quad\text{given by}\quad
        \bigoplus_{\alpha\in Q_1}\omega_{\alpha}\mapsto
        \sum_{\alpha\in Q_1}\omega_{\alpha}
    \end{align*}
    is a bimodule isomorphism.
\end{proposition}

\begin{proof}
    First of all, $\phi$ is clearly a surjective bimodule homomorphism since $\{d\alpha:\alpha\in Q_1\}$ generates $\OmegaoneD$ by Proposition~\ref{prop:OmegaoneD.relations}. Let us now show that $\phi$ is injective. To this end, assume that
    \begin{align*}
        \phi\paraa{\bigoplus_{\alpha\in Q_1}\omega_\alpha}=0\equivalent
        \sum_{\alpha\in Q_1}\omega_{\alpha}=0.
    \end{align*}
    Evaluating at $D_{\beta,q}\in\D_2$ gives
    \begin{align*}
        \omega_{\beta}(D_{\beta,q}) = 0
    \end{align*}
    implying that $\omega_{\beta}(\d)=0$ for all $\d\in\D_2$ since $\omega_{\beta}(D_{\gamma,p})=0$ if $\gamma\neq\beta$. Hence, $\omega_{\beta}=0$ for all $\beta\in Q_1$, showing that $\phi$ is indeed injective. 
\end{proof}

\noindent 
The above result tells us that the module differential forms splits as
\begin{align*}
    \OmegaoneD = \bigoplus_{\alpha\in Q_1}\angles{d\alpha}
\end{align*}
as a $\KQ$-bimodule. We will proceed to derive further properties of the modules $\angles{da}$, but let us first study the cohomology of $\Omega_{\D_2}$.

\subsection{The cohomology of $\Omega_{\D_2}$}

The cohomology of $\Omega_{\g}$ is defined in the standard way as for differential graded algebras
\begin{align*}
    H^k(\Omega_{\g}) = \frac{\ker(d_k)}{\im(d_{k-1})}
\end{align*}
describing closed $k$-forms that are not exact. In the following we will show that if $\K$ is a field of characteristic $0$ then $H^k(\Omega_{\D_2})=0$ for $k\geq 1$. However, it turns out that one can prove a slightly more general result for $\naturals$-graded algebras.

To this end, let $A$ be a $\naturals$-graded algebra over $\K$ and write
\begin{align*}
    A = \bigoplus_{k\geq 0}A_k
\end{align*}
as a sum of graded components. In this setting, one can define the so called Euler derivation $\dh\in\Der(A)$ satisfying
\begin{align*}
    \dh(a) = ka
\end{align*}
for $a\in A_k$. One can then prove the following result.

\begin{proposition}\label{prop:Omega.cohomology}
    Let $A$ be a $\naturals$-graded algebra over a field $\K$ of characteristic zero
    \begin{align*}
        A = \bigoplus_{k\geq 0}A_k
    \end{align*}
    and let $\g\subseteq\Der(A)$ be a Lie algebra over $\K$. If the Euler derivation is an element of $\g$ and $d A_0=0$ then $H^k(\Omega_\g)=0$ for $k\geq 1$.
\end{proposition}

\begin{proof}
    Write 
    \begin{align*}
        A = \bigoplus_{k\geq 0}A_k
    \end{align*}
    and let $\dh$ denote the Euler derivation, satisfying $\dh(a)=ka$ for $a\in A_k$. It follows that $\iota_{\dh}(da) = da(\dh)=\dh(a)=ka$ and the Lie derivative becomes
    \begin{align*}
        &L_{\dh}(a) = (d\circ\iota_{\dh}+\iota_{\dh}\circ d)(a)=\iota_{\dh}(da)=ka\\
        &L_{\dh}(da) = (d\circ\iota_{\dh}+\iota_{\dh}\circ d)(da) = kda.
    \end{align*}
    Hence, if $a_l\in A_{k_l}$ for $0\leq l\leq n$ then
    \begin{align*}
        L_{\dh}(a_0da_1da_2\cdots da_n) = (k_0+\cdots+k_n)a_0da_1da_2\cdots da_n,
    \end{align*}
    since $L_{\dh}(\omega\eta)=L_{\dh}(\omega)\eta+\omega L_{\dh}(\eta)$, and it is clear that $\Omega^n_\g$ can be generated by such elements. Let $0\neq\omega\in\Omega^n_{\g}$ (for $n\geq 1$) such that $L_{\dh}(\omega)=m\omega$, and it follows that $m\geq 1$ since $dA_0=0$. If $d\omega=0$ then
    \begin{align*}
        m\omega = L_{\dh}(\omega) = (d\circ\iota_{\dh}+\iota_{\dh}\circ d)(\omega)
         = d(\iota_{\dh}\omega)
    \end{align*}
    and since $\K$ has characteristic $0$ it follows that
    \begin{align*}
        \omega = \frac{1}{m}d(\iota_{\dh}\omega)
    \end{align*}
    showing that every closed form is exact. One concludes that $H^k(\Omega_\g)=0$ for $k\geq 1$.
\end{proof}

\noindent
The conditions in Proposition~\ref{prop:Omega.cohomology} is satisfied for the Lie algebra $\D_2$ and, furthermore, one can easily compute $H^0(\Omega_{\D_2})$.

\begin{corollary}
    If $Q=(Q_0,Q_1)$ is a quiver then
    \begin{align*}
        &H^0(\Omega_{\D_2}) = \K\angles{Q_0}\\
        &H^k(\Omega_{\D_2}) = 0\quad\text{for }k\geq 1.
    \end{align*}
\end{corollary}

\begin{proof}
    The path algebra $\KQ$ is $\naturals$-graded with $A_0=\K\angles{Q_0}$ and since $\d u=0$ for all $\d\in\D_2$ and $u\in Q_0$ it follows that $dA_0=0$. Moreover, the Euler derivation $\dh$ is an element of $\D_2$ since
    \begin{align*}
        \dh = \sum_{\alpha\in Q_1}D_{\alpha,\alpha}\in\D_2,
    \end{align*}
    and it follows from Proposition~\ref{prop:Omega.cohomology} that $H^k(\Omega_{\D_2})=0$ for $k\geq 1$. 

    Now, assume that $da=0$ and write
    \begin{align*}
        a = a_0 + \sum_{p\in\Pp(Q)}\lambda_p p
    \end{align*}
    with $a_0\in\K\angles{Q_0}$ and $\lambda_p\in\K$. Since $da_0=0$ one finds that
    \begin{align*}
        0 = da(\dh) = \sum_{p\in\Pp(Q)}\lambda_p dp(\dh)
        =\sum_{p\in\Pp(Q)}\lambda_p l(p)p
    \end{align*}
    and since $\P(Q)$ is a basis for $\KQ$ one necessarily has that $\lambda_p l(p)=0$ for all $p\in\Pp(Q)$ implying that $\lambda_p=0$ since $l(p)\geq 1$ and the characteristic of $\K$ is zero. One concludes that 
    \begin{equation*}
        H^0(\D_2) = \{a\in\KQ:da = 0\} = \K\angles{Q_0}.\qedhere
    \end{equation*}
\end{proof}

\subsection{Separating quivers}

\noindent
Comparing with differential geometry, where each vector bundle corresponds to a finitely generated projective module over the algebra of functions, it is natural to ask if $\OmegaoneD$ is a projective $\KQ$-module.
It turns out that there is an easy condition on the quiver that is equivalent to $\OmegaoneD$ being projective as a $\KQ$-bimodule (cf. Corollary~\ref{cor:Q.separating.equiv.projective}). 
This condition is based on the notion of separating edges which we introduce in the following definition.

\begin{definition}\label{def:separating.edge}
    Let $Q=(Q_0,Q_1)$ be a quiver. An edge $\alpha\in Q_1$ is called a separating edge if the bimodule homomorphism
    \begin{align*}
        \theta_\alpha: P_{\alpha}&\to\Hom_{\K}\paraa{s(\alpha)\KQ t(\alpha),\KQ}\\
        p\otimesK q&\mapsto \theta_\alpha(p\otimesK q)(r)=prq
    \end{align*}
    is injective. An edge that is not a separating edge is called a non-separating edge. Moreover, the quiver is called a separating quiver if each $\alpha\in Q_1$ is a separating edge.
\end{definition}

\noindent
Although it will turn out that many interesting quivers are separating (e.g. acyclic quivers), there are simple examples of quivers that are not separating. 

\begin{example}
    Let $Q$ be the quiver
    \[
    \begin{tikzcd}
        \bullet\ar[loop,swap,"\alpha"]
    \end{tikzcd}
    \]
    Then $\KQ=\K[\alpha]$ is the polynomial algebra in one variable, and any elementary tensor of $P_{\alpha}$ is of the form $\alpha^n\otimesK\alpha^m$, $n,m\geq0$. For $m\neq n$, the nonzero element 
    \[
    \alpha^n\otimesK\alpha^m-\alpha^m\otimesK\alpha^n\in P_{\alpha}
    \]
    satisfies
    \[
    \theta_\alpha(\alpha^n\otimesK\alpha^m-\alpha^m\otimesK\alpha^n)(\mid)=\theta_\alpha(\alpha^n\otimesK\alpha^m-\alpha^m\otimesK\alpha^n)(\alpha^k)=0,
    \]
    and thus $0\neq\alpha^n\otimesK\alpha^m-\alpha^m\otimesK\alpha^n\in\ker(\theta_\alpha)$. It follows that $Q$ is not a separating quiver.
\end{example}

\noindent
At first glance, Definition~\ref{def:separating.edge} might seem unnatural, but it is easy to see that it is in fact equivalent to the map
\begin{align*}
    \phi_{\alpha}:P_{\alpha}&\to \angles{d\alpha}\\
    p\otimesK q &\mapsto p(d\alpha)q
\end{align*}
being a bimodule isomorphism. Hence, since $P_{\alpha}$ is a projective bimodule, it follows that 
\begin{align*}
    \OmegaoneD = \bigoplus_{\alpha\in Q_1}\angles{d\alpha}
\end{align*}
is a projective bimodule if $Q$ is a separating quiver. In the following, we will show that this is in fact an equivalence and, furthermore, find an equivalent condition on the quiver itself. Due to the decomposition above, 
we start by showing that $\alpha$ is a separating edge if and only if $\angles{d\alpha}$ is a projective bimodule. 

Recall that $B_{\alpha}=\{(p,q):t(p)=s(\alpha),s(q)=t(\alpha)\}$ and that $\{p\otimesK q:(p,q)\in B_{\alpha}\}$ is a basis of $P_{\alpha}$. Thus, define the $\K$-linear map $\phi_{\alpha}:P_{\alpha}\to\angles{d\alpha}$ by setting
\begin{align*}
    \phi_{\alpha}(p\otimesK q) = p(d\alpha)q
\end{align*}
for $(p,q)\in B_{\alpha}$ and extending it linearly to $P_{\alpha}$; it follows that $\phi_{\alpha}$ is a surjective bimodule homomorphism.

To study the projectivity of $\angles{d\alpha}$, we consider the exact sequence of bimodules
\begin{align}\label{eq:Pa.da.exact.sequence}
        0\longrightarrow \ker(\phi_{\alpha})\longrightarrow P_{\alpha}
        \overset{\phi_{\alpha}}{\longrightarrow} \angles{d\alpha}\longrightarrow 0.
\end{align}
Clearly, if $\angles{d\alpha}$ is projective then the sequence splits and $P_{\alpha}\simeq\ker(\phi_{\alpha})\oplus\angles{d\alpha}$. Conversely, if the sequence splits then $\angles{d\alpha}$ is projective since it is a direct summand of the projective module $P_{\alpha}$. Hence, the splitting of \eqref{eq:Pa.da.exact.sequence} is equivalent to $\angles{d\alpha}$ being projective. Next, let us show that the splitting of the sequence forces $\phi_{\alpha}$ to be injective. To improve the readability of the proof, let us start with the following lemma. 

\begin{lemma}\label{lemma:idempotent.uu.otimes.vv}
    Let $Q=(Q_0,Q_1)$ be a finite quiver and let $u,v\in Q_0$. If $e\in u\KQ u\otimesK v\KQ v$ with $e^2=e$ then $e=0$ or $e=u\otimesK v$.
\end{lemma}

\begin{proof}
    The algebra $A=u\KQ u\otimesK v\KQ v$ is graded by total path length, i.e. 
    \begin{align*}
        l(p\otimesK q)=l(p)+l(q)   
    \end{align*}
    and one writes 
    \begin{align*}
        e = e_0+e_1+\cdots e_N 
    \end{align*}
    for some $N\geq 0$, where $e_k$ is a homogeneous element with $l(e)=k$. From $e^2=e$ it follows that $e_0^2=e_0$ which implies that $e_0=0$ or $e_0=u\otimesK v$ since every element of degree zero can be written as $\lambda u\otimesK v$ for some $\lambda\in\K$.

    First, consider the case when $e_0=0$. Assume that there exists a smallest integer $d\geq 1$ such that $e_d\neq 0$, which implies that all terms in $e^2$ has degree at least $2d>d$. However, this contradicts $e^2=e$ and we conclude that no such $d$ can exist. Hence, $e_0=0$ implies that $e=0$. 

    Next, consider the case when $e_0=u\otimesK v$ and define $f=u\otimesK v-e$ implying that $f^2=f$ since $e(u\otimesK v)=(u\otimesK v)e=e$. Hence, $f$ is an idempotent with vanishing terms in degree zero, which implies by the previous argument that $f=0$ and we conclude that $e=u\otimesK v$. 
\end{proof}

\begin{lemma}\label{lemma:ker.phi.zero.direct.sum}
    If $P_{\alpha}=\ker(\phi_\alpha)\oplus M$ then $\ker(\phi_{\alpha})=0$.
\end{lemma}

\begin{proof}
    Set $p=s(\alpha)\otimesK t(\alpha)$ such that $P_{\alpha}=\KQ^ep$ and
    assume that there exists a bimodule $M$ such that $P_{\alpha}=\ker(\phi_{\alpha})\oplus M$. Furthermore, let 
    \begin{align*}
        \pi:P_{\alpha}\to P_{\alpha}   
    \end{align*}
    be the projection onto $\ker(\phi_{\alpha})$ such that $\pi\circ\pi=\pi$ and $\im(\pi)=\ker(\phi_{\alpha})$. Setting $e=\pi(p)$ it follows that $e^2=e$ and $e\in p\KQ^e p$, and we note that 
    $\ker(\phi_{\alpha})=\KQ^ee$ since
    \begin{align*}
        &xe = x\pi(p) = \pi(xp) = \pi(x) = x\\
        &\phi_\alpha(ae) = \phi_\alpha(a\pi(p)) = \phi_{\alpha}(\pi(ap)) = 0
    \end{align*}
    for $x\in\ker(\phi_\alpha)$ and $a\in\KQ^e$. Next, we will show that $e=0$. 

    Since $p\KQ^ep=s(\alpha)\KQ s(\alpha)\otimesK t(\alpha)\KQ t(\alpha)$ and $e^2=e$ one can use Lemma~\ref{lemma:idempotent.uu.otimes.vv} to conclude that $e=0$ or $e=s(\alpha)\otimesK t(\alpha)=p$.
    However, since $\phi_{\alpha}(p)=d\alpha\neq 0$, implying that $\ker(\phi_\alpha)\neq P_{\alpha}$, one must necessarily have $e=0$ implying that $\ker(\phi_{\alpha})=0$. 
\end{proof}

\noindent
The above results tells us that if the sequence \eqref{eq:Pa.da.exact.sequence} splits then $\phi_{\alpha}$ is a bimodule isomorphism. Conversely, it is clear that if $\phi_{\alpha}$ is a bimodule isomorphism, then the sequence splits. Thus, $\angles{d\alpha}$ is projective if and only if $\phi_{\alpha}$ is a bimodule isomorphism. Let us now show that projectivity of $\angles{d\alpha}$ is equivalent to $\alpha$ being a separating edge.

\begin{proposition}\label{prop:separating.equiv.projective}
    Let $Q=(Q_0,Q_1)$ be a quiver and let $\alpha\in Q_1$. Then the following are equivalent
    \begin{enumerate}
        \item $\alpha$ is a separating edge,
        \item $\angles{d\alpha}$ is projective as a $\KQ$-bimodule.
    \end{enumerate}
\end{proposition}

\begin{proof}
    As already noted, Lemma~\ref{lemma:ker.phi.zero.direct.sum} implies that projectivity of $\angles{d\alpha}$ is equivalent to the injectivity of $\phi_{\alpha}$. Next, let us show that $\ker(\phi_{\alpha})=\ker(\theta_{\alpha})$; namely,
    \begin{align*}
        \ker(\phi_{\alpha}) 
        &= \{\sum_{(p,q)\in B_{\alpha}}\lambda_{(p,q)}p\otimesK q:\sum_{(p,q)\in B_{\alpha}}\lambda_{(p,q)}p(d\alpha)q=0\}\\
        & = \{\sum_{(p,q)\in B_{\alpha}}\lambda_{(p,q)}p\otimesK q:
        \sum_{(p,q)\in B_{\alpha}}\lambda_{(p,q)}prq=0\quad\forall r\in\Ppara\}
    \end{align*}
    which equals $\ker(\theta_{\alpha})$, implying that $\angles{d\alpha}$ is projective if and only if $\ker(\theta_{\alpha})=0$, i.e. if $\alpha$ is a separating edge. 
\end{proof}

\begin{corollary}\label{cor:Q.separating.equiv.projective}
    If $Q$ is a quiver then $\OmegaoneD$ is a projective $\KQ$-bimodule if and only if $Q$ is a separating quiver. Consequently, if $\OmegaoneD$ is projective then 
    \begin{align*}
        \OmegaoneD\simeq\bigoplus_{\alpha\in Q_1}P_{\alpha}.
    \end{align*}
\end{corollary}

\begin{proof}
    By Proposition~\ref{prop:Omega.decomp.dalpha}, $\OmegaoneD$ decomposes as
    \begin{align*}
        \OmegaoneD=\bigoplus_{\alpha\in Q_1}\angles{d\alpha}
    \end{align*}
    implying that $\OmegaoneD$ is projective if and only if $\angles{d\alpha}$ is projective for all $\alpha\in Q_1$. By Proposition~\ref{prop:separating.equiv.projective} this is equivalent to $\alpha$ being a separating edge for all $\alpha\in Q_1$. Moreover, if $\angles{d\alpha}$ is projective, then \eqref{eq:Pa.da.exact.sequence} splits and 
    Lemma~\ref{lemma:ker.phi.zero.direct.sum} implies that $\angles{d\alpha}\simeq P_{\alpha}$ for all $\alpha\in Q_1$.
\end{proof}

\noindent
Having shown that separating quivers precisely correspond to those quivers where $\OmegaoneD$ is a projective bimodule, we ask if one may characterize the property of being a separating quiver in terms of the underlying graph? We start by showing that edges that are not separating has to be part of a cycle.

\begin{proposition}\label{prop:alpha.not.cycle.separating}
    Let $Q=(Q_0,Q_1)$ be a quiver. If $\alpha\in Q_1$ is not part of a cycle then $\alpha$ is a separating edge. 
\end{proposition}

\begin{proof}
    Let
    \begin{align*}
        x = \sum_{(p,q)\in B_{\alpha}}\lambda_{(p,q)}p\otimesK q\in P_{\alpha}
    \end{align*}
    and assume that $x\in\ker(\theta_\alpha)$, which implies that
    \begin{align*}
        \sum_{(p,q)\in B_{\alpha}}\lambda_{(p,q)}p\alpha q=0.
    \end{align*}
    If $\alpha$ is not part of a cycle then $\alpha\notin p$ and $\alpha\notin q$ for all $(p,q)\in B_\alpha$ (otherwise there exists a cycle containing $\alpha$). Thus, all paths in the sum are pairwise distinct, implying that $\lambda_{(p,q)}=0$ for all $(p,q)\in B_{\alpha}$, since $\P(Q)$ is a basis of $\KQ$. We conclude that $\ker(\theta_\alpha)=0$ and that $\alpha$ is a separating edge. 
\end{proof}

\begin{corollary}
    If $Q$ is a finite acyclic quiver then $\OmegaoneD$ is a projective $\KQ$-bimodule. 
\end{corollary}

\begin{proof}
    If $Q$ is an acyclic quiver then no edge is part of a cycle. Proposition~\ref{prop:alpha.not.cycle.separating} then gives that the quiver is separating, implying that $\OmegaoneD$ is a projective bimodule by Corollary~\ref{cor:Q.separating.equiv.projective}.
\end{proof}

\noindent
Proposition~\ref{prop:alpha.not.cycle.separating} tells us that non-separating edges can only occur when the edge is part of a cycle. As already mentioned, we will show that edges that are not separating will be parts of a unique cycle. To prove this, we need a few concepts and results from the combinatorics of words (see e.g. \cite{l:combinatorics.on.words}). Let us start by recalling the period of a word.

\begin{definition}
Let $Q$ be a quiver and let 
\begin{equation*}
    p=\alpha_1\alpha_2\cdots\alpha_n    
\end{equation*}
be a path of length $n\geq 1$. An integer $1\leq N\leq n$ is called a period of $p$ if $\alpha_{i+N}=\alpha_i$ for every $1\leq i\leq n-N$. The smallest such integer is called the minimal period of $p$.
\end{definition}

\noindent
Note that if $l(p)=n$ then $p$ always has period $n$; however, the minimal period might be smaller.

Let $Q=(Q_0,Q_1)$ be a quiver and let $\alpha\in Q_1$ such that $\alpha$ is part of a cycle. The idea in the following is to show that if one can construct a path parallel to the edge $\alpha$ with an arbitrarily long period then $\alpha$ is a separating edge. Since $\alpha$ is part of a cycle, there exist arbitrarily long paths parallel to $\alpha$ but, the obstruction occurs when $\alpha$ is part of a unique cycle in which case the period of the path is bounded. As soon as $\alpha$ is part of at least two distinct cycles, one may construct paths of arbitrarily long periods.

Let us start with the following result, showing that paths of long periods separate short paths.

\begin{lemma}\label{lemma:no.periods.pi.qi.equal}
    Let $N>0$ be an integer and let $w\in\Pp(Q)$ be a path such that if $n$ is a period of $w$ then $n>N$. If $p_1,p_2,q_1,q_2\in\P(Q)$ such that 
    \begin{align*}
        \max\{l(p_1),l(p_2),l(q_1),l(q_2)\}\leq N   
    \end{align*}
    and
    \begin{align*}
        p_1wq_1 = p_2wq_2\neq 0
    \end{align*}
    then $(p_1,q_1)=(p_2,q_2)$.
\end{lemma}

\begin{proof}
    First, assume that $l(p_1)<l(p_2)$. Since $p_1wq_1=p_2wq_2\neq 0$ there exists $r\in\Pp(Q)$ such that $p_2=p_1r$ giving
    \begin{align*}
        wq_2 = rwq_2
    \end{align*}
    implying that $w$ has period $1\leq l(r)<N$. This contradicts the assumption that $w$ has no period less than $N$ and we conclude that $l(p_1)\geq l(p_2)$. An analogous argument shows that $l(p_1)>l(p_2)$ is not possible. Thus, $l(p_1)=l(p_2)$ which implies that $p_1=p_2$ and consequently that $q_1=q_2$.
\end{proof}

\begin{lemma}\label{lemma:if.nonperiodic.separating}
    Let $Q=(Q_0,Q_1)$ be a quiver and let $\alpha\in Q_1$. If for each integer $N>0$ there exists $p\in\Ppara$ such that if $n$ is a period of $p$ then $n>N$, then $\alpha$ is a separating edge.
\end{lemma}

\begin{proof}
    Let $0\neq x\in\P_\alpha$ and write
    \begin{align*}
        x = \sum_{(p,q)\in B_{\alpha}} \lambda_{(p,q)} p\otimesK q
    \end{align*}
    with $\lambda_{p,q}\in\K$, and let
    \begin{align*}
        N = \max_{(p,q)\in B_{\alpha}:\lambda_{(p,q)}\neq 0}
        \{l(p),l(q)\}.
    \end{align*}
    Now, let $w$ be a path with no periods $\leq N$ (which exists by assumption) and consider
    \begin{align*}
        \theta_{\alpha}(x)(w) = \sum_{(p,q)\in B_{\alpha}} \lambda_{(p,q)} pwq.
    \end{align*}
    By using Lemma~\ref{lemma:no.periods.pi.qi.equal} one concludes that
    \begin{align*}
        pwq = p'wq' \equivalent (p,q)=(p',q')
    \end{align*}
    implying that all paths in $\theta_\alpha(x)(w)$ are linearly independent. One concludes that $\theta_\alpha(x)(w)\neq 0$ implying that $\ker(\theta_{\alpha})=0$ and that $\alpha$ is a separating edge.
\end{proof}

\noindent
Finally, we prove a result that allows us to construct cycles of arbitrarily long periods starting from two distinct cycles based at the same vertex. 

\begin{lemma}\label{lemma:pm.qm.no.period.N}
    Let $p,q\in\Pp(Q)$ such that $s(p)=t(p)=s(q)=t(q)$ and $[p,q]\neq 0$. For each integer $N>0$ there exists an integer $m>0$ such that $p^m q^m$ has no period less than or equal to $N$.  
\end{lemma}

\begin{proof}
    Let $l_0=\min(l(p),l(q))$ and set 
    \begin{align*}
        m = N!\paraa{l(p)+l(q)-\gcd(l(p),l(q))}
    \end{align*}
    and assume that $p^mq^m$ has a period $d\leq N$. 
    Since $d$ divides $m$ it follows that the length of $p^m$ is a multiple of $d$. Hence, since $p^mq^m$ has period $d$, after reaching the end of $p^m$ the beginning edges of $q^m$ must conicide with those of $p^m$, i.e. $p^m$ and $q^m$ agree up to the first $m\cdot l_0$ edges.
    Since
    \begin{align*}
        m\cdot l_0 = N!l_0\paraa{l(p)+l(q)-\gcd(l(p),l(q))}
        \geq l(p)+l(q)-\gcd(l(p),l(q))
    \end{align*}
    one concludes from the Fine-Wilf theorem (see e.g. \cite[Prop 1.3.5]{l:combinatorics.on.words}) that $p$ and $q$ are powers of the same cycle $c$ which implies that $[p,q]=0$, contradicting the assumption. Hence, $p^mq^m$ has no period less than or equal to $N$.  
\end{proof}

\noindent
At this point, we have all the necessary results to prove that non-separating edges are characterized by the fact that they are part of a unique cycle.

\begin{theorem}\label{thm:separating.equiv.cycle}
    Let $Q=(Q_0,Q_1)$ be a quiver and let $\alpha\in Q_1$. Then 
    $\alpha$ is a non-separating edge if and only if there exists $s\in\P(Q)$ such that $c=\alpha s$ is a cycle and
    \begin{align*}
        \Ppara=\{c^n\alpha:n\geq 0\}.
    \end{align*}
\end{theorem}

\begin{proof}
    Let us start by assuming that $\alpha$ is a non-separating edge. If $\alpha$ is not a separating edge then Proposition~\ref{prop:alpha.not.cycle.separating} implies that $\alpha$ is part of a cycle. Hence, there exists $s\in\Pp(Q)$ such that $c=\alpha s$ is a cycle at $s(\alpha)$. Now, let $r\in\Ppara$, which implies that $c'=rs$ is a cycle at $s(\alpha)$. 
    
    If $[c,c']\neq 0$ then Lemma~\ref{lemma:pm.qm.no.period.N} implies than one can construct a cycle $\hat{c}_N$ which has no period less than $N$ for any integer $N>0$; consequently, $\hat{c}_N\alpha$ is an element of $\Ppara$ which has no period less than $N$. Lemma~\ref{lemma:if.nonperiodic.separating} then implies that $\alpha$ is a separating edge, which contradicts the assumption. Hence, $[c',c]=0$.

    Since $[c,c']=0$ and $cc',c'c\neq 0$, the cycles $c,c'$ have to be powers of the same cycle $\tilde{c}$ (see e.g. \cite[Prop 1.3.2]{l:combinatorics.on.words}), i.e. $c=\tilde{c}^{n_1}$ and $c'=\tilde{c}^{n_2}$. Since $l(\alpha s)\leq l(rs)$ one necessarily has $n_2\geq n_1$ giving
    \begin{align*}
        rs = \tilde{c}^{n_2-n_1}\tilde{c}^{n_1}=\tilde{c}^{n_2-n_1}\alpha s
    \end{align*}
    implying that $r=\tilde{c}^{n_2-n_1}\alpha$. Hence, every element $r\in\Ppara$ is of the form $\tilde{c}^n\alpha$ for some $n\geq 0$.
    To prove the converse statement, assume that $s\in\P(Q)$ such that $c=\alpha s$ is a cycle and $\Ppara = \{c^n\alpha:n\geq 0\}$. Then $s(s)=t(\alpha)$ and $t(s)=s(\alpha)=s(c)$, so
    \begin{align*}
        \theta_\alpha(c\otimesK s-s(\alpha)\otimesK sc)(c^n\alpha) = 
        c^{n+1}\alpha s - s(\alpha)c^n\alpha sc=c^{n+1}c-c^ncc=0
    \end{align*}
    for any $n\geq0$. Thus $0\neq c\otimesK s-s(\alpha)\otimesK sc\in\ker\theta_\alpha$, and hence $\alpha$ is a non--separating edge.
\end{proof}

\noindent
Recall that two cycles are distinct if they are not powers of the same cycle. Theorem~\ref{thm:separating.equiv.cycle} tells us that a separating edge is either not part of any cycle, or part of at least two distinct cycles. Namely, if $\alpha\in Q_1$ is part of two distinct cycles $c_1,c_2$ then there does not exist a cycle $c$ such that $c_1=c^n$ and $c_2=c^m$,  contradicting $\Ppara=\{c^n\alpha:n\geq 0\}$.

Let us summarize the results in this section and formulate the property of being a separating quiver in the following equivalent ways.

\begin{corollary}\label{cor:equiv.sep.proj.cycles}
    Let $Q=(Q_0,Q_1)$ be a quiver. Then the following are equivalent
    \begin{enumerate}
        \item $Q$ is a separating quiver,
        \item $\OmegaoneD$ is a projective bimodule,
        \item every $\alpha\in Q_1$ is part of either 0 or at least 2 distinct cycles.
    \end{enumerate}    
\end{corollary}

\subsection{Left module structure of $\OmegaoneD$}

Let us briefly discuss the left module structure of $\OmegaoneD$, as we will later study the dimension of this module in relation to its dual (as a right module).

First of all, it turns out that one can easily decompose $P_{\alpha}$ as a direct sum of projective left modules.

\begin{proposition}\label{prop:Omegaone.left.module.decomposition}
    If $Q=(Q_0,Q_1)$ is a quiver and $\alpha\in Q_1$ then
    \begin{align*}
        P_{\alpha}\simeq\hspace{-2mm}
        \bigoplus_{p:\,s(p)=t(\alpha)}\hspace{-2mm}
        \KQ s(\alpha)
        =\paraa{\KQ s(\alpha)}^{\dim t(\alpha)\KQ}
    \end{align*}
    as left $\KQ$-modules.
\end{proposition}

\begin{proof}
    Define
    \begin{align*}
        \phi:
        \bigoplus_{p:\,s(p)=t(\alpha)}\hspace{-2mm}
        \KQ s(\alpha)
        &\to
        P_{\alpha}\\
        \bigoplus_{p:\,s(p)=t(\alpha)}\hspace{-2mm}
        c_{\alpha}^p s(\alpha)&\mapsto\hspace{-2mm}
        \sum_{p:\,s(p)=t(\alpha)}c_{\alpha}^p s(\alpha)\otimesK t(\alpha)p
    \end{align*}
    Then $\phi$ is a well-defined left module homomorphism since if
    \begin{align*}
        x = \bigoplus_{p:\,s(p)=t(\alpha)}\hspace{-2mm} c_{\alpha}^p s(\alpha)=0
    \end{align*}
    then $c_{\alpha}^p s(\alpha)=0$ for each $p$, implying that $\phi(x)=0$. Moreover, $\phi$ is clearly surjective. Let us show that $\phi$ is injective. To this end, assume that 
    \begin{align*}
        x = \bigoplus_{p:\,s(p)=t(\alpha)}\hspace{-2mm} c_{\alpha}^p s(\alpha)\qand
        \phi(x) 
        = \hspace{-3mm}\sum_{p:\,s(p)=t(\alpha)}\hspace{-2mm}
        c_{\alpha}^p\otimesK p = 0
    \end{align*}
    and write
    \begin{align*}
        c_{\alpha}^p = \sum_{q:\, t(q)=s(\alpha)}
        \lambda_{p,q}q\qquad (\lambda_{p,q}\in\K)
    \end{align*}
    giving
    \begin{align*}
        0=\sum_{q:\, t(q)=s(\alpha)}
        \sum_{p:\,s(p)=t(\alpha)}\hspace{-2mm}
        \lambda_{p,q}q\otimesK p
        =\sum_{(q,p)\in B_{\alpha}}\lambda_{p,q}q\otimesK p.
    \end{align*}
    Since $\{q\otimesK p:(q,p)\in B_{\alpha}\}$ is a basis for $P_{\alpha}$ we conclude that $\lambda_{p,q}=0$ for all $(p,q)\in B_{\alpha}$ implying that $x=0$. Hence, $\phi$ is a left module isomorphism.
\end{proof}

\noindent
Proposition~\ref{prop:Omegaone.left.module.decomposition}, together with Corollary~\ref{cor:Q.separating.equiv.projective}, implies that one may decompose $\OmegaoneD$ as a left module in the following way
\begin{align*}
    \OmegaoneD \simeq \bigoplus_{\alpha\in Q_1}\bigoplus_{p:\,s(p)=t(\alpha)} \KQ s(\alpha)
    \simeq \bigoplus_{\alpha\in Q_1} (\KQ s(\alpha))^{\dim(t(\alpha)\KQ)}
\end{align*}
if $Q$ is a separating quiver.

\subsection{Endomorphisms of $\OmegaoneD$}\label{sec:endomorphisms}

\noindent
In the following, we will study bimodule endomorphisms of $\OmegaoneD$ via the decomposition into the projective modules $P_{\alpha}$. Thus, if $\OmegaoneD$ is projective then
\begin{align*}
    \OmegaoneD\simeq \bigoplus_{\alpha\in Q_1}P_\alpha
\end{align*}
implying that 
\begin{align}
    \Hom_{\KQ}(\OmegaoneD,\OmegaoneD)\simeq
    \bigoplus_{\alpha,\beta\in Q_1}\Hom_{\KQ}(P_\alpha,P_\beta),
\end{align}
where $\Hom_{\KQ}(M,N)$ denotes the set of bimodule homomorphisms from $M$ to $N$. It turns out that nonzero endomorphisms are "local" in the sense that there are no nonzero homomorphisms from $P_{\alpha}$ to $P_{\beta}$ unless $\beta$ lies in a path parallel to $\alpha$ (cf. Lemma~\ref{lemma:Hom.Pa.Pb.parallel}). 

Let us start by identifying $\Hom_{\KQ}(P_{\alpha},P_{\beta})$ with a vector subspace of $\KQ\otimesK \KQ$.

\begin{proposition}\label{prop:Homab.iso.stKQ}
    If $Q=(Q_0,Q_1)$ is a quiver and $\alpha,\beta\in Q_1$ then
    \begin{align*}
        \phi:\Hom_{\KQ}(P_\alpha,P_\beta)&\to s(\alpha)\KQ s(\beta)\otimesK t(\beta)\KQ t(\alpha)\\
        f &\mapsto f(s(\alpha)\otimesK t(\alpha))
    \end{align*}
    is a vector space isomorphism.
\end{proposition}

\begin{proof}
    By definition, $\phi(f)\in \KQ s(\beta)\otimesK t(\beta)\KQ$ and, moreover,  
    \begin{align*}
        f(s(\alpha)\otimesK t(\alpha)) 
        =f(s(\alpha)^2\otimesK t(\alpha)^2) = s(\alpha)f(s(\alpha)\otimesK t(\alpha))t(\alpha)
    \end{align*}
    showing that $\phi(f)\in s(\alpha)\KQ s(\beta)\otimesK t(\beta)\KQ t(\alpha)$. Since $f$ is a bimodule homomorphism, it is completely determined by $\phi(f)$ since $f(ps(\alpha)\otimesK t(\alpha)q)=p\phi(f)q$. Hence, if $\phi(f)=0$ then 
    \begin{align*}
        f(s(\alpha)\otimes t(\alpha)) = 0\implies f=0
    \end{align*}
    showing that $\phi$ is injective. 

    To show that $\phi$ is surjective, let $f_{\alpha}\in s(\alpha)\KQ s(\beta)\otimesK t(\beta)\KQ t(\alpha)$. Since a basis of $P_\alpha$ is given by $\{p\otimesK q:(p,q)\in B_{\alpha}\}$, one can construct a linear map $f\in \Hom_{\K}(P_\alpha,P_\beta)$ by setting
    \begin{align*}
        f(p\otimesK q) = pf_{\alpha}q
        \qquad \text{for }(p,q)\in B_{\alpha},
    \end{align*}
    and one notes that $f$ is a bimodule homomorphism satisfying
    \begin{align*}
        \phi(f) = f(s(\alpha)\otimesK t(\alpha)) = s(\alpha)f_{\alpha}t(\alpha) = f_{\alpha},
    \end{align*}
    showing that $\phi$ is indeed surjective. 
\end{proof}

\noindent
As already mentioned, endomorphisms of $\OmegaoneD$ has a certain local behavior where nonzero homomorphisms from $P_{\alpha}$ to $P_{\beta}$ only exist if $\beta$ lies in a path parallel to $\alpha$. This is made precise in the following result. 

\begin{lemma}\label{lemma:Hom.Pa.Pb.parallel}
    If $\alpha,\beta\in Q_1$ such that $\beta\notin\Qpalpha$ then $\Hom_{\KQ}(P_\alpha,P_{\beta})=0$. 
\end{lemma}

\begin{proof}
    Since $\beta\notin\Qpalpha$ if follows that either $s(\alpha)\KQ s(\beta)=0$ or $t(\beta)\KQ t(\alpha)=0$; otherwise $\beta$ would necessarily be in a path parallel to $\alpha$. This shows that 
    \begin{align*}
        s(\alpha)\KQ s(\beta)\otimesK t(\beta)\KQ t(\alpha)=0   
    \end{align*}
    implying that $\Hom_{\KQ}(P_\alpha,P_\beta)=0$ by Proposition~\ref{prop:Homab.iso.stKQ}.
\end{proof}

\begin{corollary}\label{cor:Hom.zero.no.parallel}
    Let $\alpha\in Q_1$. If $\alpha$ is not a bypass then $\Hom_{\KQ}(P_\alpha,P_{\beta})=0$ for all $Q_1\ni\beta\neq\alpha$.
\end{corollary}

\noindent
The above result implies that for $\alpha\in Q_1$ that is not a bypass, the only nonzero homomorphisms lie in $\Hom_{\KQ}(P_{\alpha},P_{\alpha})$, which is identified (via Proposition~\ref{prop:Homab.iso.stKQ}) with
\begin{align*}
    \Hom_{\KQ}(P_\alpha,P_\alpha)\simeq 
    s(\alpha)\KQ s(\alpha)\otimesK t(\alpha)\KQ t(\alpha)
    =\K\angles{s(\alpha)\otimesK t(\alpha)}
\end{align*}
since there are no nontrivial elements in $s(\alpha)\KQ s(\alpha)$ or $t(\alpha)\KQ t(\alpha)$ if $\alpha$ is not a bypass. Hence, if $f\in\Hom_{\KQ}(P_\alpha,P_\alpha)$ then there exists $\lambda_f\in\K$ such that
\begin{align*}
    f(d\alpha) = \lambda_f d\alpha.
\end{align*}

\begin{corollary}
    If $Q=(Q_0,Q_1)$ is a nonpassing quiver then
    \begin{align*}
        \Hom_{\KQ}(\OmegaoneD,\OmegaoneD)\simeq \K^{|Q_1|}.
    \end{align*}
\end{corollary}

\begin{proof}
    If $Q$ is a nonpassing quiver then $\alpha$ is not a bypass for any $\alpha\in Q_1$. Corollary~\ref{cor:Hom.zero.no.parallel} then implies that every $f\in\Hom_{\KQ}(\OmegaoneD,\OmegaoneD)$ satisfies
    \begin{align*}
        f(d\alpha) = \lambda_{\alpha}d\alpha
    \end{align*}
    for some $\lambda_{\alpha}\in\K$. Conversely, every such choice of $\lambda_{\alpha}\in\K$ defines an element $f\in\Hom_{\KQ}(\OmegaoneD,\OmegaoneD)$ by Proposition~\ref{prop:Homab.iso.stKQ}.
\end{proof}

\noindent
In Section~\ref{sec:derivation.based.calculi} we recalled that the set of connections is parametrized by elements $A\in\Hom_{\KQ,\K}(\OmegaoneD\times\D_2,\OmegaoneD)$ giving 
\begin{align*}
    \nabla = \nabla^0 + A
\end{align*}
for an arbitrary fixed connection $\nabla^0$. Thus, given $\d\in\D_2$ we note that $A(\cdot,\d)\in\Hom_{\KQ}(\OmegaoneD,\OmegaoneD)$ for which the above results apply.  

\begin{lemma}\label{lemma:wedgeA.Qpar}
    Let $Q$ be a separating quiver and let $A\in\Hom_{\KQ,\K}(\OmegaoneD\times\D_2,\OmegaoneD)$. If $\alpha\in Q_1$ and $(\wedge A)(d\alpha)=0$ then
    \begin{align*}
        A(d\alpha,D_{\beta,p}) = 0
    \end{align*}
    for all $D_{\beta,p}\in\D_2$ such that $\beta\notin\Qpalpha$.
\end{lemma}

\begin{proof}
    Since $\beta\notin\Qpalpha$, it follows from Lemma~\ref{lemma:Hom.Pa.Pb.parallel} that 
    \begin{align*}
        \Hom_{\KQ}(P_\alpha,P_{\beta}) = 0
    \end{align*}
    implying that $A(d\alpha,D_{\gamma,q})$ has no component in $P_\beta$ and, consequently, that
    \begin{align*}
        A(d\alpha,D_{\gamma,q})(D_{\beta,p})=0
    \end{align*}
    for all $D_{\gamma,q}\in\D_2$. Since $(\wedge A)(d\alpha)=0$ it follows that
    \begin{align*}
        0 = A(d\alpha,D_{\beta,p})(D_{\gamma,q})-A(d\alpha,D_{\gamma,q})(D_{\beta,p})
        =A(d\alpha,D_{\beta,p})(D_{\gamma,q})
    \end{align*}
    for all $D_{\gamma,p}\in\D_2$ which shows that $A(d\alpha,D_{\beta,p})=0$.
\end{proof}

\noindent
Let us apply the above result to the case when $\alpha\in Q_1$ is not a bypass. 
Thus, if $(\wedge A)(d\alpha)=0$ then $A(d\alpha,\d)=0$ unless $\d=D_{\alpha,\alpha}$, i.e there exists $\lambda_{\alpha}\in\K$ such that
\begin{align*}
    A(d\alpha,\d) = 
    \begin{cases}
        \lambda_{\alpha}d\alpha &\text{if }\d=D_{\alpha,\alpha}\\
        0 &\text{otherwise}
    \end{cases}
\end{align*}
implying that for a nonpassing quiver, $A\in\ker(\wedge)$ is determined by
\begin{align*}
    A(d\alpha,D_{\beta,\beta}) = \delta_{\alpha\beta}\lambda_{\alpha}d\alpha.
\end{align*}
for $\lambda_{\alpha}\in\K$. Conversely, any $A\in\Hom_{\KQ}(P_\alpha,P_\alpha)$ defined by the above relation satisfies $\wedge A=0$.

\section{Bimodule connections}

\noindent
Let us now turn our attention to finding connections on the bimodule of differential forms. We shall in the following assume that the quiver is separating, implying that $\OmegaoneD$ is a projective bimodule and that
\begin{align*}
    \OmegaoneD \simeq \bigoplus_{\alpha\in Q_1} P_{\alpha}.
\end{align*}
On projective modules, one can define the Grassmann connection by starting from a connection on the free module and then composing it with the projection. To this end, 
one defines a bimodule connection on the free bimodule $\KQ^e=\KQ\otimesK \KQ^{\textrm{op}}$ as
\begin{align*}
    \nabla_{\d}(a\otimesK b) = (\d a)\otimesK b + a\otimesK(\d b), 
\end{align*}
together with the projection
\begin{align*}
    p_{\alpha}:\KQ^e&\to\KQ^e\\
    a\otimesK b&\mapsto as(\alpha)\otimesK t(\alpha)b.
\end{align*}
Setting
\begin{align*}
    p = \bigoplus_{\alpha\in Q_1}p_{\alpha}:(\KQ^e)^{|Q_1|}\to(\KQ^e)^{|Q_1|}.
\end{align*}
one defines the Grassmann connection on $\OmegaoneD$ as
\begin{align*}
    \nablaz_{\d} = p\circ\nabla:\OmegaoneD\to\OmegaoneD
\end{align*}
giving
\begin{align}\label{eq:grassmann.conn}
    \nablaz_{\d}a(d\alpha)b = (\d a)(d\alpha)b + a(d\alpha)(\d b)
\end{align}
for $a,b\in\KQ$ and $\d\in\D_2$. In particular, we note that $\nablaz_{\d}d\alpha=0$ for all $\alpha\in Q_1$ and $\d\in\D_2$, implying that the Grassmann connection is torsion free since
\begin{align*}
    \paraa{\nablaz_{\d}d\alpha}(\d')-\paraa{\nablaz_{\d'}d\alpha}(\d)- (d^2\alpha)(\d,\d')
    =  0
\end{align*}
for all $\alpha\in Q_1$ and $\d,\d'\in\D_2$.
Consequently, all bimodules connections on $\OmegaoneD$ are obtained as
\begin{align*}
    \nabla = \nablaz + A
\end{align*}
where $A\in\Hom_{\KQ,\K}(\OmegaoneD\times\g,\OmegaoneD)$ and $\nabla$ is torsion free if and only if $\wedge A=0$.

\subsection{Bimodule forms}

Let us now consider bimodule forms on $\OmegaoneD$. Since 
\begin{align*}
    \OmegaoneD\simeq \bigoplus_{\alpha\in Q_1}P_\alpha
\end{align*}
one may write a bimodule form $g:\OmegaoneD\otimesK\OmegaoneD\to\KQ$ as
\begin{align*}
    g = \bigoplus_{\alpha,\beta\in Q_1}g_{\alpha\beta}\qquad g_{\alpha\beta}:P_\alpha\otimesK P_\beta\to \KQ.
\end{align*}
We note that 
\begin{align*}
    g(d\alpha p,qd\beta) = s(\alpha)g(d\alpha p,qd\beta)t(\beta)\implies
    g(d\alpha p,qd\beta)\in s(\alpha)\KQ t(\beta).
\end{align*}
Conversely, since $\{p\otimesK q:(p,q)\in B_\alpha\}$ is a basis of $P_\alpha$, given 
\begin{align*}
    g_{(\alpha,p)(q,\beta)}\in s(\alpha)\KQ t(\beta)
\end{align*}
for $\alpha,\beta\in Q_1$ and $p\in S_{\alpha}$, $q\in T_{\beta}$, one defines a bimodule form $g$ as
\begin{align*}
    g(r d\alpha p,qd\beta s)
    =rg_{(\alpha,p)(q,\beta)}s.
\end{align*}
For instance, for any $c\in\KQ$ one can choose $g_{(\alpha,p)(q,\beta)}=\alpha pcq\beta$ giving
\begin{align*}
    g_c(rd\alpha p,qd\beta s) = r\alpha pcq\beta s. 
\end{align*}
If $c\in Z(\KQ)$ then $g_c$ is a strong bimodule form. Let us now present a few examples to illustrate the different properties of bimodule forms.

\begin{example}
    Let $Q$ be the quiver $A_2:$
    \begin{displaymath}
        \begin{tikzcd}
            \bullet \arrow[r, "\alpha"] & \bullet
        \end{tikzcd}
    \end{displaymath}
    for which $\OmegaoneD=\{\lambda d\alpha:\lambda\in\K\}$. Defining a bimodule form $g$ as
    \begin{align*}
        g(\lambda d\alpha,\mu d\alpha) = \lambda\mu\alpha
    \end{align*}
    if follows that $g$ is both left and right nondegenerate.
\end{example}

\begin{example}
    Let $Q$ be the Kronecker quiver
    \begin{displaymath}
        \begin{tikzcd}
            \bullet
            \arrow[r, bend left=25, "\alpha"]
            \arrow[r, bend right=25, swap, "\beta"]
            & \bullet
        \end{tikzcd}
    \end{displaymath}   
    giving $\OmegaoneD=\{\lambda d\alpha + \mu d\beta:\lambda,\mu\in\K\}$. Define a bimodule form $g$ as
    \begin{align*}
        &g(d\alpha,d\alpha) = \alpha\qquad 
        g(d\alpha,d\beta) = \beta \\
        &g(d\beta,d\alpha) = g(d\beta,d\beta) = 0.
    \end{align*}
    It is easy to check that $\gh_L$ is injective but $\gh_R$ is not injective since $\gh_R(d\beta)=0$. Hence, $g$ is left nondegenerate but not right nondegenerate.
\end{example}

\begin{example}
    Let $Q$ be the quiver $A_3$
    \begin{displaymath}
        \begin{tikzcd}
            \bullet 
            \arrow[r, "\alpha"] 
            & \bullet 
            \arrow[r, "\beta"] 
            & \bullet 
        \end{tikzcd}
    \end{displaymath}
    implying that $\OmegaoneD$ is the $\K$-linear span of $\{d\alpha,d\beta,(d\alpha)\beta,\alpha d\beta\}$. The bimodule form $g$ given by
    \begin{align*}
        &g(d\alpha,d\alpha) = 0\quad g(d\alpha,d\beta)=\alpha\beta\quad 
        g(d\alpha,\alpha d\beta) = 0\\
        &g(d\beta,d\alpha) = g(d\beta,d\beta) = g(d\beta,\alpha d\beta) = 0\\
        &g\paraa{(d\alpha)\beta,d\alpha} =
        g\paraa{(d\alpha)\beta,d\beta}=g\paraa{(d\alpha),\alpha d\beta} = 0,
    \end{align*}
    which is of the form $g_c$ for $c=\mid$, is a nonzero bimodule form which is neither left nondegenerate nor right nondegenerate.
\end{example}

\begin{example}
    The previous example can be generalized to quivers which has either a source or a sink. For instance, acyclic quivers have at least one source and at least one sink. Thus, let $Q=(Q_0,Q_1)$ be a quiver, let $v\in Q_0$ be a source and let $\alpha_0\in Q_1$ be such that $s(\alpha_0)=v$. The bimodule form $g_c$ defined by $g_c\paraa{d\alpha p,q d\beta}=\alpha p c q \beta$ for arbitrary $c\in\KQ$ (as above) gives
    \begin{align*}
        g(d\beta p,d\alpha_0) = \beta p c \alpha_0 = 0
    \end{align*}
    for all $\beta\in Q_1$ and $p\in\P(Q)$, since $s(\alpha_0)$ is a source and every nonzero term in $\beta p c$ is a path of length at least one. Hence, $g$ is left degenerate since $\gh_L(d\alpha_0)=0$.

    Similarly, if there exists a sink in the quiver then one shows that $g$ is right degenerate. In particular, $g_c$ is both left and right degenerate for an acylic quiver.
\end{example}

\begin{example}\label{ex:loop.invertible.metric}
    Let $Q$ be the quiver
    \[
    \begin{tikzcd}
        \bullet\ar[loop,swap,"\alpha"]
    \end{tikzcd}
    \]
    giving $\KQ=\K[\alpha]$ and $\OmegaoneD=\K[\alpha]d\alpha$. Furthermore, we note that $\alpha^n d\alpha = (d\alpha)\alpha^n$ and the quiver is not separating since $\alpha$ is part of a unique cycle. Since the quiver is not separating we have to take care in defining a bimodule form. However, since $\{\alpha^n d\alpha: n\geq 0\}$ is a vector space basis of $\OmegaoneD$ one can define
    \begin{align*}
        g(\alpha^n d\alpha,\alpha^m d\alpha) = \alpha^{n+m}.
    \end{align*}
    Let us now show that $g$ is right invertible, and start by computing
    \begin{align*}
        \gh_R(\alpha^n d\alpha)(\alpha^md\alpha)
        = g(\alpha^n d\alpha,\alpha^m d\alpha) = \alpha^{m+n}.
    \end{align*}
    It follows that $\gh_R$ is injective since
    \begin{align*}
        \gh_R(p(\alpha)d\alpha) = 0\implies
        \gh_R(p(\alpha)d\alpha)(d\alpha) = 0\implies
        p(\alpha) =0\implies p(\alpha)d\alpha = 0,
    \end{align*}
    for $p(\alpha)\in\K[\alpha]$. Now, let $\phi\in(\OmegaoneD)^\ast$. The element $\phi$ is completely determined by $\phi(d\alpha)\equiv\phi_0\in\K[\alpha]$, and we note that
    \begin{align*}
        \gh_R(\phi_0 d\alpha)(d\alpha) = \phi_0 = \phi(d\alpha)
    \end{align*}
    implying that $\gh_R(\phi_0d\alpha)=\phi$. This shows that $\gh_R$ is bijective. Similarly, one can show that $\gh_L$ is bijective. Hence, $g$ is left and right invertible. 
\end{example}

\noindent In Example~\ref{ex:loop.invertible.metric} we have illustrated that certain quivers may have invertible bimodule forms. However, this is not the situation in general; namely, for acyclic quivers, where $\OmegaoneD$ and $(\OmegaoneD)^\ast$ are finite dimensional vector spaces, one can show that their dimensions differ, implying that no invertible bimodule form exists.

\begin{proposition}
    If $Q=(Q_0,Q_1)$ is an acyclic quiver and $|Q_1|>0$ then 
    \begin{align*}
        \dim_{\K}((\OmegaoneD)^\ast)-\dim_{\K}(\OmegaoneD)\geq |\Pp(Q)|.
    \end{align*}
\end{proposition}

\begin{proof}
    An acylic quiver is separating which implies that 
    \begin{align*}
        \OmegaoneD =  \bigoplus_{\alpha\in Q_1}P_{\alpha}
    \end{align*}    
    and, as previously noted, a set of basis elements (as a vector space) can be given as
    \begin{align*}
        M_1 = \{(p,\alpha,q): \alpha\in Q_1,\, (p,q)\in B_{\alpha}\}
    \end{align*}
    which is in bijection with 
    \begin{align*}
        M_2 = \{(p,r):p\in\P(Q),\,r\in\Pp(Q),\,t(p)=s(r)\}
    \end{align*}
    via $(p,\alpha,q)\mapsto (p,\alpha q)$. As a left module, $P_{\alpha}$ decomposes as
    \begin{align*}
        P_{\alpha}\simeq\hspace{-2mm}
        \bigoplus_{p:\,s(p)=t(\alpha)}\hspace{-2mm}
        \KQ s(\alpha)
    \end{align*}
    implying that the dual $P_\alpha^\ast$ decomposes as
    \begin{align*}
        P_{\alpha}\simeq\hspace{-2mm}
        \bigoplus_{p:\,s(p)=t(\alpha)}\hspace{-2mm}
        s(\alpha)\KQ
    \end{align*}
    as a right module, and consequently
    \begin{align*}
        (\OmegaoneD)^\ast = \bigoplus_{\alpha\in Q_1}
        \bigoplus_{p:\, s(p)=t(\alpha)}
        s(\alpha)\KQ
    \end{align*}
    for which a set of basis elements can be identified with
    \begin{align*}
        N_1 = \{(\alpha,p,q):\alpha\in Q_1,\,p,q\in\P(Q),\, s(p)=t(\alpha),\,s(q)=s(\alpha)\}
    \end{align*}
    which is in bijection with
    \begin{align*}
        N_2 = \{(r,q):r\in\Pp(Q),q\in\P(Q),\, s(r)=s(q)\}
    \end{align*}
    via $(\alpha,p,q)\mapsto (\alpha p,q)$.

    Let us now define
    \begin{align*}
        \phi:M_2&\to N_2\\
        (p,r)&\mapsto (pr,p),
    \end{align*}
    and it directly follows that $\phi$ is injective. Now, assume that $Q_1\neq\emptyset$ and let $p\in\Pp(Q)$. Then $(p,p)\in N_2$ but $(p,p)\notin\im(\phi)$ since
    \begin{align*}
        \phi(q,r) = (p,p)\implies (qr,q) = (p,p)
    \end{align*}
    which is impossible since $r\in\Pp(Q)$. Hence, $|N_2|-|M_2|\geq |\Pp(Q)|$, implying that $\dim_{\K}((\OmegaoneD)^\ast)-\dim_{\K}(\OmegaoneD)\geq |\Pp(Q)|$.
\end{proof}

\begin{corollary}
    If $Q=(Q_0,Q_1)$ is an acyclic quiver with $Q_1\neq\emptyset$ then there exists no invertible bimodule form on $\OmegaoneD$.
\end{corollary}

\subsection{Levi-Civita connections}

In this section, we will provide examples of Levi-Civita connections on quivers; i.e. torsion free connections on $\OmegaoneD$ that are compatible with a bimodule form. It is a general feature of noncommutative geometry that unless further assumptions are made, the Levi-Civita connection is in general not unique, even in the case of invertible forms (see e.g. \cite{a:levi-civita.class.nms,ah:nc.levi.civita} in the context of derivation based calculi). However, in the example we consider below, there is a weaker form of nondegeneracy which turns out to be sufficient to guarantee that the Levi-Civita connection is unique. To this end, let us introduce weakly nondegenerate bimodule forms.  

\begin{definition}
    Let $Q=(Q_0,Q_1)$ be a quiver. A bimodule form $g$ on $\OmegaoneD$ is called weakly nondegenerate if for each $\alpha\in Q_1$ there exists $p,q\in \P(Q)$ and $\beta\in Q_1$ such that 
    \begin{align*}
        g(d\alpha p,qd\beta)\neq 0\quad\text{or}\quad g(d\beta p,qd\alpha)\neq 0.
    \end{align*}
\end{definition}

\noindent
Let us now consider Levi-Civita connections on nonpassing quivers, compatible with bimodule forms of the type $g_c$ for $c\in\KQ$. Recall from Section~\ref{sec:endomorphisms} that, in this case, the structure of $\Hom_{\KQ}(\OmegaoneD,\OmegaoneD)$ is particularly simple, and every torsion free connection can be written as
\begin{align*}
    \nabla_{D_{\beta,\beta}}d\alpha 
    = \nabla^0_{D_{\beta,\beta}}d\alpha + A(d\alpha,D_{\beta,\beta})
    = \delta_{\alpha\beta}\lambda_{\alpha}d\alpha
\end{align*}
for $\lambda_{\alpha}\in\K$. This allows us to prove the following result.

\begin{proposition}
    If $Q=(Q_0,Q_1)$ is a nonpassing quiver then there exists a torsion free connection compatible with $g_c$ if and only if 
    there exists $\tilde{c}\in\K\angles{Q_0}$ such that $g_c=g_{\tilde{c}}$. Moreover, if $g_c$ is weakly nondegenerate then the connection is uniquely given by 
    \begin{align}
        \nabla_{D_{\beta,\beta}}d\alpha = \delta_{\alpha\beta}d\alpha
    \end{align}
    for $\alpha\in Q_1$ and $D_{\beta,\beta}\in\D_2$.
\end{proposition}

\begin{proof}
    Recall that the Grassmann connection $\nabla^0$ is defined by
    \begin{align*}
        \nabla^0_\d a(d\alpha) b=(\d a)(d\alpha)b+a(d\alpha)(\d b)
    \end{align*}
    for $a,b\in\KQ$ and $\d\in\D_2$, and that any torsion free connection is on the form $\nabla^0+A$, where $A\in\Hom_{\KQ,\K}(\OmegaoneD\times\D_2,\OmegaoneD)$ such that $\wedge A=0$. Since $Q$ is nonpassing, $\{\d_\gamma = D_{\gamma,\gamma}:\gamma\in Q_1\}$ is a basis of $\D_2$, and any $A\in \Hom_{\KQ,\K}(\OmegaoneD\times\D_2,\OmegaoneD)$ is of the form
    \begin{align*}
        A(d\alpha,\d_\gamma)=\lambda_\alpha\delta_{\alpha\gamma}d\alpha,
    \end{align*}
    for $\alpha,\gamma\in Q_1$ and $\lambda_\alpha\in\K$, implying that $\wedge A=0$.
    
    Let us first show that if $c\in\K\angles{Q_0}$ then there exists a torsion free connection compatible with $g_c$. To this end let $\nabla$ be the (torsion free) connection defined by
    \begin{align*}
        \nabla_{\d_\gamma}d\alpha = \delta_{\alpha\gamma}d\alpha.
    \end{align*}
    This connection is compatible with $g_c$ since
    \begin{align*}
        \d_\gamma g_c&(d\alpha p,qd\beta)-g_c(\nabla_{\d_\gamma}d\alpha p,qd\beta)-g_c(d\alpha p,\nabla_{\d_\gamma}qd\beta)\\
        &=(\d_\gamma\alpha)pcq\beta+\alpha p(\d_\gamma c)q\beta+\alpha pcq(\d_\gamma\beta)
        -\delta_{\alpha\gamma}\alpha pcq\beta-\delta_{\beta\gamma}\alpha pcq\beta\\
        &=\alpha p(\d_{\gamma}c)q\beta =0
    \end{align*}
    due to the assumption that $c\in \K\angles{Q_0}$.

    Conversely, assume that $c\in\KQ$ and that $\nabla$ is a torsion free connection compatible with $g_c$. Since $Q$ is a nonpassing quiver, there exists $\lambda_{\alpha}\in\K$ such that 
    \begin{align*}
        \nabla_{\d_\gamma}d\alpha = \lambda_{\alpha}\delta_{\alpha\gamma}d\alpha.
    \end{align*}
    Write
    \begin{align*}
        c = \sum_{v\in Q_0}\lambda_v v + \sum_{r\in\Pp(Q)}\mu_r r
    \end{align*}
    for $\lambda_v,\mu_r\in \K$. If $\mu_r\alpha p r q\beta=0$ for all $p,q\in\P(Q)$, $r\in\Pp(Q)$ and $\alpha,\beta\in Q_1$ then one may choose
    \begin{align*}
        \tilde{c} = \sum_{v\in Q_0}\lambda_v v\in\K\angles{Q_0}
    \end{align*}
    giving $g_c = g_{\tilde{c}}$. Next, assume that there exists $r_0\in\Pp(Q)$, $\alpha,\beta\in Q_1$ and $p,q\in\P(Q)$ such that $\mu_{r_0}\alpha p r_0 q \beta\neq 0$. One notes that $\alpha\neq\beta$ and $\alpha,\beta\notin r_0$ since the quiver is acyclic, implying that $\alpha p (\d_\alpha c) q \beta=\alpha p (\d_\beta c) q \beta=0$.
     
    Compatibility with $g_c$ gives
    \begin{align*}
        0 &= \d_\gamma g_c(d\alpha p,qd\beta)-g_c(\nabla_{\d_\gamma}d\alpha p,qd\beta)-g_c(d\alpha p,\nabla_{\d_\gamma}qd\beta)\\
        &=(\d_\gamma\alpha)pcq\beta+\alpha p(\d_\gamma c)q\beta+\alpha pcq(\d_\gamma\beta)-\lambda_\alpha\chi_{\alpha,\alpha}(\d_\gamma)\alpha pcq\beta-\lambda_\beta\chi_{\beta,\beta}(\d_\gamma)\alpha pcq\beta\\
        &=\delta_{\alpha\gamma}\alpha pcq\beta+\alpha p(\d_\gamma c)q\beta+\delta_{\beta\gamma}\alpha pcq\beta-\delta_{\alpha\gamma}\lambda_\alpha\alpha pcq\beta-\delta_{\beta\gamma}\lambda_\beta\alpha pcq\beta\\
        &=\alpha p(\d_\gamma c)q\beta+\delta_{\alpha\gamma}(1-\lambda_\alpha)\alpha pcq\beta+\delta_{\beta\gamma}(1-\lambda_\beta)\alpha pcq\beta.    
    \end{align*}
    For $\gamma=\alpha$ and $\gamma=\beta$ one finds that $\lambda_{\alpha}=1$ and $\lambda_{\beta}=1$, respectively. Now, let $\gamma\in r_0$ (implying that $\gamma\neq\alpha,\beta$). The compatibility equation then gives
    \begin{align*}
        0=\alpha p(\d_\gamma c)q\beta 
        = \sum_{r\in\Pp(Q)}\mu_r\alpha p (\d_{\gamma} r) q\beta
        = \mu_{r_0}\alpha p r_0 q\beta
        +\hspace{-4mm}\sum_{r\in\Pp(Q),r\neq r_0}\hspace{-4mm}\mu_r\alpha p (\d_{\gamma} r) q\beta
    \end{align*}
    implying that $\mu_{r_0}\alpha p r_0 q\beta=0$ since $\P(Q)$ is a basis for $\KQ$ and the quiver is acyclic, which contradicts the assumption that $\mu_{r_0}\alpha p r_0 q\beta\neq 0$. We conclude that no such $r_0\in\Pp(Q)$ exists, which implies that $g_c=g_{\tilde{c}}$ for 
    \begin{equation*}
        \tilde{c} = \sum_{v\in Q_0}\lambda_v v\in\K\angles{Q_0},
    \end{equation*}
    completing the proof of the first part of the statement.

    Next, let us show that the compatible connection is unique if $c\in\K\angles{Q_0}$ and $g_c$ is weakly nondegenerate. 
    Since $g_c$ is weakly nondegenerate, for each $\alpha\in Q_1$ there exists $\beta\in Q_1$ and $p,q\in\P(Q)$ such that either $g_c(d\alpha p,qd\beta)\neq 0$ or $g_c(d\beta p,qd\alpha)\neq 0$. We have previously concluded that if there exists a compatible connection and $g_c(d\alpha p,qd\beta)\neq 0$ then $\lambda_{\alpha}=\lambda_{\beta}=1$ (with $g_c(d\beta p,qd\alpha)\neq 0$ giving the same conclusion), implying that the unique torsion free connection compatible with $g_c$ is given by
    \begin{equation*}
        \nabla_{\d_\gamma}d\alpha = \delta_{\alpha\gamma}d\alpha. \qedhere
    \end{equation*}
\end{proof}

\noindent
Although the above result concerns a particular class of quivers, it illustrates the fact that there is in general a non-trivial theory of Levi-Civita connections on path algebras. It remains an interesting question to completely classify quivers and bimodule forms for which Levi-Civita connections exist. 

\section{Examples}

\noindent
Let us end by presenting a few examples to illustrate some of the contents developed in this paper. To better understand the structure of the Lie algebra $\D_2$ in the examples below, let us recall the following result
\begin{proposition}[\cite{h:hochschild.finite.dim.alg}]
    Let $Q=(Q_0,Q_1)$ be an acyclic quiver and let $H^1(\KQ)$ denote the first Hochschild cohomology of $\KQ$.
    Then
    \begin{align*}
        \dim H^1(\KQ) = 1-|Q_0|+
        \sum_{\alpha\in Q_1}\nu(\alpha)
    \end{align*}
    where $\nu(\alpha)=\dim s(\alpha)\KQ t(\alpha)$.
\end{proposition}

\noindent
For nonpassing quivers it holds that $\dim s(\alpha)\KQ t(\alpha)=1$ for all $\alpha\in Q_1$ implying that
\begin{align*}
    \dim H^1(\KQ) = 1-|Q_0|+|Q_1|.
\end{align*}
Furthermore, if $Q$ is a tree then $|Q_0|=|Q_1|+1$ implying that $\dim H^1(\KQ)=0$.

The first Hochschild cohomology of $\KQ$ describes the outer derivations modulo the inner ones. Hence, if $Q$ is a tree, e.g. $Q=A_4$
\[
\begin{tikzcd}[column sep=2em]
1 \arrow[r,"\alpha"] & 2 \arrow[r,"\beta"] & 3 \arrow[r,"\gamma"] & 4
\end{tikzcd}
\]
then there are no outer derivations, and one finds that in the case of $A_4$
\begin{align*}
    D_{\alpha,\alpha}(a) = [v_1,a]\qquad
    D_{\beta,\beta}(a) = [a,v_3+v_4]\qquad
    D_{\gamma,\gamma}(a) = [a,v_4]
\end{align*}
by using \eqref{eq:inner.der.in.D2}. In fact, as shown in \cite{gl:structure.hochschild.path.algebras}, $\dim H^1(\KQ)>0$ if and only if $Q$ contains an unoriented cycle.

\subsection{A nonpassing quiver}

\[
\begin{tikzcd}[column sep=2em, row sep=0.1em]
& 2 \arrow[dr,"\alpha_2"] & \\
1 \arrow[ur,"\alpha_1"] \arrow[dr,"\beta_1"'] && 4 \\
& 3 \arrow[ur,"\beta_2"'] &
\end{tikzcd}
\]

\noindent
\begin{align*}
    \D_2 &= \K\angles{D_{\alpha_1,\alpha_1},D_{\alpha_2,\alpha_2},D_{\beta_{1},\beta_1},D_{\beta_2,\beta_2}}\\
    \OmegaoneD &= \angles{d\alpha_1}\oplus \angles{d\alpha_2}\oplus\angles{d\beta_1}\oplus\angles{d\beta_2}\\
    &=\K\angles{d\alpha_1,(d\alpha_1)\alpha_2}\oplus
    \K\angles{d\alpha_2,\alpha_1 d\alpha_2}\oplus
    \K\angles{d\beta_1,(d\beta_1)\beta_2}\oplus
    \K\angles{d\beta_2,\beta_1 d\beta_2}\\
    \Omega_{\D_2}^2 &= \K\angles{d\alpha_1 d\alpha_2,d\beta_1 d\beta_2}
\end{align*}

\noindent 
Since $\dim\D_2=4$ one knows that $\Omega_{\D_2}^k=0$ for $k>4$; however, since the longest path in the quiver is of length 2 it follows that $\Omega_{\D_2}^k=0$ for $k>2$, and we note that $\dim\Omega_{\D_2}^1 = 8$ and $\dim\Omega_{\D_2}^2=2$.

Furthermore, the above quiver is a nonpassing quiver, giving $\dim H^1(\KQ)=1-4+4=1$. Using \eqref{eq:inner.der.in.D2} one finds that 
\begin{align*}
    &D_{\alpha_2,\alpha_2}(a) = D_{\alpha_1,\alpha_1}(a)+[v_2,a]\\
    &D_{\beta_1,\beta_1}(a) = -D_{\alpha_1,\alpha_1}(a)+[v_1,a]\\
    &D_{\beta_2,\beta_2}(a) = -D_{\alpha_1,\alpha_1}(a)+[v_1+v_3,a]
\end{align*}
for $a\in\KQ$, giving $[D_{\alpha_1,\alpha_1}]=[D_{\alpha_2,\alpha_2}]$ and $[D_{\beta_1,\beta_1}]=[D_{\beta_2,\beta_2}]=-[D_{\alpha_1,\alpha_1}]$ in $H^1(\KQ)$. Moreover, $[D_{\alpha_1,\alpha_1}]\neq 0$ since assuming that $D_{\alpha_1,\alpha_1}(a)=[b,a]$ with
\begin{align*}
    b = \sum_{i=1}^4 c_iv_i + \lambda_1\alpha_1+\lambda_2\alpha_2 + \mu_1\beta_1 + \mu_2\beta_2
    +\tau_1\alpha_1\alpha_2 + \tau_2\beta_1\beta_2
\end{align*}
for $c_i,\lambda_1,\lambda_2,\mu_1,\mu_2,\tau_1,\tau_2\in\K$, gives 
\begin{align*}
    &\alpha_1\alpha_2 = D_{\alpha_1,\alpha_1}(\alpha_1\alpha_2) = [b,\alpha_1\alpha_2]
    = (c_1-c_4)\alpha_1\alpha_2\\
    &0 = D_{\alpha_1,\alpha_1}(\beta_1\beta_2) = [b,\beta_1\beta_2] = (c_1-c_4)\beta_1\beta_2
\end{align*}
showing that $D_{\alpha_1,\alpha_1}$ is not an inner derivation and, hence, $[D_{\alpha_1,\alpha_1}]\neq 0$ in $H^1(\KQ)$. It follows that $D_{\alpha_2,\alpha_2},D_{\beta_1,\beta_1},D_{\beta_2,\beta_2}$ are outer derivations as well. 

\subsection{An acyclic quiver with a bypass}

\[
\begin{tikzcd}[column sep=1.5em, row sep=0.5em]
& 2 \arrow[dr,"\beta"] & \\
1 \arrow[ur,"\alpha"] \arrow[rr,"\gamma"'] & & 3 
\end{tikzcd}
\]

\begin{align*}
    \D_2 &= \K\angles{D_{\alpha,\alpha},D_{\beta,\beta},D_{\gamma,\gamma},D_{\gamma,\alpha\beta}}\\
    \OmegaoneD &= \angles{d\alpha}\oplus\angles{d\beta}\oplus\angles{d\gamma}\\
    &= \K\angles{d\alpha,(d\alpha)\beta}\oplus\K\angles{d\beta,\alpha d\beta}\oplus
    \K\angles{d\gamma}\\
    \Omega_{\D_2}^2 &= \K\angles{d\alpha d\beta}.
\end{align*}

\noindent
In this example we note that $\gamma$ is a bypass and that $\dim\Omega_{\D_2}^1 = 5$ and $\dim\Omega_{\D_2}^2=1$, as well as $\Omega_{\D_2}^k=0$ for $k>2$ since the longest path in the quiver has length 2. One finds that $\dim H^1(\KQ) = 1-3+1+1+2 = 2$ and one can easily check that
\begin{align*}
    0\neq[D_{\alpha,\alpha}]=[D_{\beta,\beta}]=-[D_{\gamma,\gamma}]
\end{align*}
in $H^1(\KQ)$, and that $[D_{\alpha,\alpha}]$ and $[D_{\gamma,\alpha\beta}]\neq 0$ are linearly independent. Hence,
\begin{align*}
    H^1(\KQ) = \K\angles{[D_{\alpha,\alpha}],[D_{\gamma,\alpha\beta}]},
\end{align*}
and it follows from the above that $D_{\alpha,\alpha},D_{\beta,\beta},D_{\gamma,\gamma},D_{\gamma,\alpha\beta}$ are outer derivations.

\subsection{A cyclic separating quiver}

\[
\begin{tikzcd}
1
\arrow[out=140, in=220, loop, looseness=5, "\alpha"']
\arrow[out=40, in=320, loop, looseness=5, "\beta"]
\end{tikzcd}
\]

\noindent
The path algebra of the above quiver is the free noncommutative algebra on the generators $\alpha,\beta$. A basis is given by words in $\alpha,\beta$ and $s(p)=t(p)=v_1=\mid$ for all $p\in\P(Q)$. Hence
\begin{align*}
    \D_2 &= \K\angles{D_{\alpha,a},D_{\beta,a}:a\in\KQ}\\
    \Omega_{\D_2}^1 &= \angles{d\alpha}\oplus\angles{d\beta}\\
    &=\bigg\{\sum_{i}a_i(d\alpha)b_i:a_i,b_i\in\KQ\bigg\}\oplus
    \bigg\{\sum_{i}a_i(d\beta)b_i:a_i,b_i\in\KQ\bigg\}.
\end{align*}
It follows from  Corollary~\ref{cor:equiv.sep.proj.cycles} that $Q$ is a separating quiver implying that $\angles{d\alpha},\angles{d\beta}$ are projective modules. In fact, let us show that $\angles{d\alpha},\angles{d\beta}$ are free $\KQ$-bimodules. 
Namely, since $s(\alpha)\KQ t(\alpha)=\KQ$ the fact that $Q$ is a separating quiver implies that 
\begin{align*}
    \tilde{\theta}_{\alpha}:\KQ\otimes\KQ&\to\angles{d\alpha}\\
    a\otimes b &\mapsto a(d\alpha)b
\end{align*}
is injective, which is equivalent to $d\alpha$ being a basis of $\angles{d\alpha}$ as a left $\KQ^e$-module. The same argument applies to $\angles{d\beta}$.

Let us now show that $\Omega_{\D_2}^n\neq 0$ for all $n>0$. To this end, let
\begin{align*}
    \omega = d\alpha d\alpha \cdots d\alpha\in\Omega_{\D_2}^n
\end{align*}
and let $w_k=\alpha^{k}\beta^{n-k}$ for $1\leq k\leq n$. It follows that
\begin{align*}
    \omega(D_{\alpha,w_1},D_{\alpha,w_2},\ldots,D_{\alpha,w_n})
    =\sum_{\sigma\in S_n} w_{\sigma(1)}w_{\sigma(2)}\cdots w_{\sigma(n)}
\end{align*}
and we note that all paths in the sum are distinct, implying that $\omega\neq 0$ and consequently that $\Omega_{\D_2}^n\neq 0$. 

\section{Summary}

\noindent
In this paper, we have developed a derivation based differential geometry for quiver path algebras, starting from a distinguished Lie algebra $\D_2$ of derivations associated with substitutions of arrows by parallel paths. The resulting differential calculus reflects the combinatorial structure of the underlying quiver in a direct way. For instance, for acyclic quivers the dimension of the calculus is determined by the length of the longest path, and its cohomology is trivial in all positive degrees.

A central result is the characterization of those quivers for which the bimodule of differential forms is projective as a bimodule over the path algebra. Introducing the notion of a \emph{separating quiver}, we have shown that $\Omega^1_{\D_2}$ is projective
if and only if every edge is either not part of any cycle or part of at least two distinct cycles. In this case,
\begin{align*}
    \Omega^1_{\D_2}
    \simeq
    \bigoplus_{\alpha\in Q_1} P_\alpha,
\end{align*}
giving an explicit decomposition into elementary projective bimodules. Thus, a purely graph theoretic property of the quiver determines an essential geometric property of the associated differential calculus.

This projective description also makes it possible to study bimodule endomorphisms, connections and metric structures. We have constructed a canonical torsion-free Grassmann connection, analyzed bimodule forms and their degeneracy properties, and shown that nontrivial acyclic quivers admit no invertible bimodule forms. Finally, for nonpassing quivers we obtain an existence criterion for Levi-Civita connections and, under a weak nondegeneracy assumption, a unique compatible torsion-free connection.

Taken together, these results show that path algebras support a nontrivial derivation based noncommutative Riemannian geometry in which algebraic, geometric and combinatorial structures are closely intertwined. They also suggest several natural directions for further work, including a more complete classification of Levi-Civita structures and extensions of the construction to related classes of algebras, such as path algebras with relations as well as Leavitt path algebras.

\bibliographystyle{alpha}
\bibliography{references}

\end{document}